\documentclass{article}
\usepackage{graphicx} 
\usepackage{amsmath}

\fontsize{12pt}{16pt}\selectfont

\usepackage{graphicx} 
\usepackage{amsmath,amssymb,amsmath}
\usepackage{enumitem,color}
\usepackage[margin=0.9in]{geometry}

\usepackage{amsthm}

\usepackage{xcolor}

\usepackage{float}

\newtheorem{ex}{Example}
\newtheorem{mydef}{Definition}
\newtheorem{rmk}{Remark}
\newtheorem{thrm}{Theorem}

\newtheorem{Proposition}{Proposition}

\title{Axiomatic Learning and Probabilistic Structures with first applications to variants of Arithmetic Machine Learning and beyond}
\author{Vlad Margarint}
\date{October 2026}

\begin{document}

\maketitle

\textbf{Abstract} We introduce the Axiomatic Learning Structures (ALS) and Axiomatic Probabilistic Structure (APS) in order to study the interactions of Machine Learning, Optimization, with research frameworks- which, we take the view, can be characterized at the core by the corresponding axioms. The ALS explores the interplays between structures that one is interested in Learning Theory such as Loss functions and the axioms of the objects and transformations involved in these functions. The APS studies the interaction between structures of interest in Probability Theory such as Gibbs distributions, equations with noise, random fields, etc. and the axioms of the corresponding objects and transformations. We also introduce in Section 6 Transformations-Axiomatic Based Learning (TABLe) on Mathematical data-sets. One feature of this approach is the ranking of transformations from a table for a certain task based on scores which reflect which transformations work better than others, similarly to when human investigators try various methods on a task. We study an initial controlled experiment. In addition, we study structure recovery from noisy data. Among others, in this context, one aim of this research direction is to study the interplay between the number of (noisy) data points, the rigidity and structure of the axioms and of the objects, the characteristics of the noise, and the structure of the landscape of the loss functions, energies, etc. In particular, in this first study, in order to exemplify the ideas, we consider loss functions which contain derivations and derivations-like transformations. As a first application of a certain ALS we study recovery in the context of polynomials with integer coefficients using standard derivatives along with other types of derivations which appear in Arithmetic: the p-derivations which satisfy \textit{different axioms} from the usual derivatives. 
To keep the same theme, as a first application of the APS we study a class of noisy (partial/mixed) derivations equations for polynomials with integer coefficients, along a spatial and 'arithmetic' direction. 
The introduced studies in this paper open problems which we hope will lead to new research directions. A part of these directions can be described as Transformations-Axiomatic Based Learning (TABLe), Data-driven Arithmetic Geometry, along with Intersections and Enumerative versions. In the last sections, to provide further potential applications of this framework, we give an overview of an initial simplified Opinion Dynamics model and its (simultaneous) study of rules of updates.

\section{Introduction}

 An Axiomatic Learning Structure (ALS) consists of a space of objects, a loss function or a sequence of loss functions, operations between the objects, and transformations applied to the objects. We consider structures with or without noise. The first ones in the context of structure recovery from noisy data, and the second one in the context of Conjecture Generation from Mathematical objects data-sets. These loss functions depend on the axioms of the operations between the objects and of the transformations we consider. 
 Minimization problems of Loss functions appear naturally in Machine Learning (ML) (see \cite{bishop2023deep}, \cite{BishopML}). In addition to ML, there are works that apply these types of minimization techniques in the field of Inverse problems (see \cite{klibanov1992inversebook}, \cite{klibanov2012approximate}). See also \cite{grabchak2026test}, \cite{mcgoff2020empirical}, \cite{farmer2024empirical} and \cite{shrotri2026neural}. The last reference studied Machine Learning methods in the novel, to the best of my knowledge, context of singular dynamics.

As proof search and the solution of well-formulated problems can become more automated, parts of the work of researchers may shift toward choosing the problems, inventing the appropriate structures, etc. See \cite{chau2025machine} for Research-level Mathematical objects data-sets. In the current work we aim to study, in one direction, recovery of structure from noisy data (we study noisy data for integer polynomials, but the perspective applies more broadly to other data sets). In a second direction, we define and investigate Transformations-Axiomatic Based Learning (TABLe) on structured mathematical data-sets with the aim to develop a research direction in which transformations may be learned. In this context, we want to study loss functions defined by their structural rules, establish recovery and estimation guarantees, and use transformations to search mathematical databases for new patterns and conjectures. In general, the $ALS$ are designed to be flexible to include operations and various types of operations over the spaces of interest. This framework can be adapted to different fields of research.

Although we will focus on this current manuscript on the recovery of structure from noisy data, the objective of this approach is multifold and it contains multiple directions. For example, one could investigate the following questions.
\begin{itemize}

\item[] 1)Recovery.
Under what conditions can the original mathematical object be recovered from noisy observations?

\item[] 2)Identifiability.
When can the underlying transformations and their axioms/rules themselves be recovered?

\item[] 3)Generalization.
When do relationships discovered in a finite collection of mathematical objects extend to larger or previously unseen classes of objects?

\item[] 4)Approximate identities.
Can small transformation-regularization errors reveal meaningful approximate mathematical structures?

\item[] 5)Conjecture generation.
Can learned transformations suggest identities of the form
\[
F(T_1(X),\ldots,T_m(X))=0
\]
that are subsequently recognized as general mathematical statements and proved independently?
\end{itemize}
These questions move the problem from  Machine Learning prediction towards structural discovery. We take the view that the axioms carry core information about the corresponding model one will consider for the data-set or would like to learn about the objects in the data. One can consider various versions of this analysis including noisy data sets, for example, coming from Biology or other Sciences, or datasets without noise like the mathematical objects data sets. We focus this work on the recovery from noisy data cases, but one can study with these ideas data-sets without noise. For the data-sets without noise one wants to learn rules which best capture the essence of the data. For this, in this work, we introduce the Transformations-Axiomatic based Learning (TABLe) approach and Conjecture Generation in Section 6 of this paper.
The approach is based on assigning transformations from a library scores based on how close they were to characterize the structure. We also perform an initial experiment in this context in the case of symbolic power of ideals geometric characterization from the algebraic characterization.
This method does not offer only the solutions or transformations that worked but also the ones that did not and how far-off where they based on the scores obtained. Furthermore, we can further take these findings and send the information to some Large Language Models to further explore the found transformations and what type of results/conjectures one might expect and then to a Lean formalization.
In addition, one can perform a uniform analysis over various choices of transformations, see the discussion on a model linked with Economics in Section 10 below.

We note that there are many mathematical objects also in the continuum whose space is described by rules. One can also study both classes of continuum and discrete mathematical structures given by rules and the corresponding objects that form them and the loss function evaluated on them. Thus, the methods we develop work for both discrete and continuous. We focus on the discrete case in this work.
In the continuous case one can apply gradient methods, which are already developed in the literature.

We consider objects $r \in R$, where $R$ is a generic mathematical structure.
The $ALS_d$ contains a loss function or a collection of such functions that we seek to minimize.  One can also consider maximization Optimization problems as this type of framework is general, however we restrict to the minimization type problems for convenience in this manuscript. In the context of Structure Recovery from noisy data, one can study minimization optimization problems such as hard constraints problems, soft constraint problems, augmented versions (see for example, \cite{veryimportantsoftandhardlagrangepaper2023augmented}), etc. We focus this work on the soft problems, and we discuss the motivations for doing so in the corresponding sections. However, we provide experiments both in the hard constraints and soft constraint formulations, as there is no restriction to study only soft problems in the context of ALS, as they are designed such that one can consider general optimization problems of various natures. 

A generic ALS contains a space of objects and operations $(R, \circ_1, \circ_2, \ldots, *_1, *_2, \ldots)$ with $*_i$ and $\circ_i$, $i \geq 1$, satisfying certain axioms along with the transformations $T_j$ which respect also some rules with respect to the operations, which we can capture (for the sake of generality) as $T_j(r_1*_ir_2)=\Phi_{i,j}(r_1, r_2, *_i),$ and $T_j(r_1\circ_ir_2)=\tilde{\Phi}_{i,j}(r_1, r_2, \circ_i), $ where $\Phi_{i,j}()$ and $\tilde{\Phi}_{i,j}()$ are functions of the generic elements $r_1, r_2 \in R$ and the operations. This perspective helps when one considers also inequalities in the constraints. 
In one section of this work, we study specific fundamental transformations which appear in various area of research, namely derivations or derivations-like transformations, but the setting developed in this manuscript is general, and captures generic spaces and transformations along with their axioms. This leads to a particular choice of the functions $\Phi_{i,j}(\cdot)$ and $\tilde{\Phi}_{i,j}(\cdot)$.  We will call them for simplicity $ALS_d$. Derivations are studied in algebraic settings and are maps which usually satisfy as axioms Linearity and Leibniz Rule $\partial(fg)=g\partial(f)+f\partial(g)$, for $f$ and $g$ objects in our space of consideration. Derivations appear in many studies and have many applications in various areas of research such as Complex Geometry, Coding theory, see for a non-exhaustive list \cite{higherderivationssurvey}, \cite{referencederivsakai1966derivations}, \cite{Algebraicfunctionfieldsandcodes}, \cite{klimek2021derivationsapplications}, \cite{derivationsapplicationfinitefields} and the references therein. See also \cite{xin2008derivationsonlattices} for lattice versions.

These loss functions depend on the choices that we make for the operations and for the notions of derivations and derivations-like operations that we work with. One of the first examples is the $ALS_d$ with a standard notion of derivative from Mathematical Analysis as a choice of derivation. The usual derivative satisfies Linearity and Leibniz Rule. This example of $ALS_d$ is assumed to be known and we will not cover it in this work.
These frameworks extend the study of loss functions with standard derivatives to more general transformations. As a particular example, in order to exemplify the ideas, in this work we study an Axiomatic Learning Structure linked with Arithmetic. This has as elements polynomials with integers coefficients along  with natural operations such as addition or multiplication, etc. and two notions of derivations, the usual derivative and the p-derivation. The latter one was introduced in \cite{Buiumoriginalpaper1} \cite{Joyal}.The p-derivations do \textit{not} satisfy Linearity and they satisfy a \textit{modified version of Leibniz rule} (see \cite{Buiumoriginalpaper1}). This is a first point in which these derivation-like objects are setting apart from the usual derivative.  A natural question is then how these properties do reflect in the structure of the landscape of the loss functions.

One direction of research when working with structure recovery from noisy data, is to consider the interaction between the structure of the axioms of transformations, the rigidity of the objects, the number of observed noisy data points, the characteristic of the noise, and the landscape and the minimization of the loss function. In general, the algebraic and arithmetic structures come with rigidity and intuitively one might need few data points to discover a rigid object. As a first application, we consider polynomials with integer coefficients but one can consider other objects such as polynomials with coefficients a subset of the integers, or on other structures along with the choice of transformations as derivations and derivations-like operations. However, the scope of this work is more general for generic objects that can be understood through a number of parameters (in the case of polynomials, these are the coefficients) and general transformations between them given by some rules/axioms and some formulas. 
In this work, we consider synthetic data generated from a ground truth integer-coefficient polynomial to which we add noise.  The fact that the ground-truth polynomials are integer coefficients is an instance of rigidity. As the objects are integer valued it is interesting to ask how does rigidity interact with the number of observed data points. The advantage of using synthetic data is that given that one knows the ground truth one can measure the recovery error, along with robustness to noise, etc. We perform a part of these experiments in this work. 
These loss functions are a hybrid of least-square term and penalty terms.  These types of objects appear in the literature when one considers  mixes of least-square and interpolation problems of the object itself or the usual derivatives and gradients (see \cite{veryimportantsoftandhardlagrangepaper2023augmented}, \cite{booksurfacereconstructionpage5322015scale}, \cite{CSimagereconstructpaperrobust}). We consider extensions of these in which the usual gradient is replaced with mixed usual and p-derivations. 

In the second part of the paper, we study a notion of Arithmetic Energy and its landscape for a family of polynomials, and we consider also analogies with Gibbs measures. In a separate direction, we introduce the Axiomatic Probabilistic structure (APS). An example of APS which we study consists of a space of objects, operations between the objects, axioms that the notions of derivations satisfy along with a notion of noise. This framework extends Buium’s work on Arithmetic Differential Equations (see \cite{Arithmeticdiffequationsbook}, \cite{Buiumoriginalpaper1}, \cite{Buiumoriginalpaper2}) to random/noisy versions. We will call them Noisy Derivations Equations and Noisy Partial Derivations Equations. We study these in Arithmetic contexts. The second category appears when one considers the spatial direction as a first component and a notion of 'arithmetic direction' as a second. We will do an initial study of these. The equations in this context appear when one is performing random search (see Section 9). This leads to equations involving (mixed) derivations, each with its own axioms, and randomness. We focus this manuscript on the  commutative case. It will be interesting to study if random versions can be obtained further for other types of approaches (see the introduction in  \cite{Arithmeticdiffequationsbook} where other types of Calculus are mentioned).

 There are previous works on exact arithmetic data (see, for example, Chapter 1 in \cite{chapter1exactarithmeticdatahe2023machine}) and roots of random polynomials in  \cite{leont2006rootsuniform}. What distinguishes our approach to the first work is that that work does not use noise, no estimation theory, and contrast to the uniform choice of noise in the second work, ours is based on a statistical estimation procedure applied to noisy data (data-driven).

We believe that it is interesting to see how the studies covered in \cite{murphy2012machine}, \cite{bishop2023deep}, \cite{BishopML}, such as Decision Theory, Information Theory, study of entropy, etc interact with  the axioms of the corresponding framework one considers. Following this, we give a non-exhaustive list of possible open problems 
In this work, we cover the axioms of some studied before frameworks but our approach is general in the sense that one allows for other axioms of frameworks that we yet might not have discovered.  Then, with those new axioms, this approach will allow to analyze how do the corresponding loss functions, etc. look and how are their landscapes, minimizers, etc. 

In order to keep the manuscript more widely accessible to the communities interested in developments in the areas of Machine Learning, we keep the exposure in this manuscript as simple as possible and focus on simple computations and initial experiments.  There are more involved theoretical and experimental directions to be further explored. 

In this work, we investigate how Loss functions considered in ML/ Optimization interact with the axioms of the transformations relevant for a framework one might consider studying. We consider some example of known transformations and axioms but we keep the framework flexible to other rules/axioms one might discover (for example, in Biology, etc). By studying the structure of the loss landscape we want to investigate also the computational complexity of corresponding tasks in the framework one considers.

\textbf{Structure of the paper: }The manuscript is divided into several sections, the first being the Introduction. In Sections 2-4, we study  $ALS_d$ based on the usual derivatives along with p-derivations and which has as objects integer-coefficients polynomials. We explore the loss landscape and some properties of these functions. We discuss applications of this $ALS_d$ in the context of data-driven Arithmetic Geometry for structure recovery at unseen noise values, as well as applications to potential data security, etc. In section 5, we study an example of other $ALS_d$ based on Hasse-Schmidt derivations (\cite{hasse1937noch}). In Section 6 we change directions, and we do not focus on structure recovery from noisy data. We introduce initial ideas of Transformations-Axiomatic based Learning (TABLe) and Conjectures generations and we study a certain subset of mathematical objects, namely symbolic power of ideals. We describe the idea of ranking transformations characterized by axioms, based on their scores when applied to the objects and we perform an initial controlled experiment.   In Section 7, we study notions of energies along with some of their properties and Gibbs measures linked with Arithmetic. In Section 8, we introduce Axiomatic Probabilistic Structures APS, and study noisy derivations equations involving usual and p-derivations emerging in random sampling, and in section 9 we perform initial experiments in these directions. In Section 10 we discuss initial applications of ALS in Economics considering various scenarios (such as looking over a set of potential transformations simultaneously along with their axioms). In the final section, we close with some open problems and other examples of ALS and APS related with Random Fields and a random analogue of the deterministic geometric arithmetic picture considered in \cite{Foundationsarithmeticgometrybook}.

\section{General view}

We devote a large part of the paper on the study of structure recovery from noisy observations using loss functions and axioms of the transformations involved. When studying some spaces of objects, and a corresponding set of natural transformations on them, one can  directly with the definition of them (if it has access to it) or with the axiomatic characterization of how the transformation acts on basic operations. The second route is also useful when one is lead to consider the study of loss functions landscapes, or perturbations around the initialization values in Optimization problems. Having access to the axioms/rules exactly gives information on the loss function/energy landscape. 

Beyond the study of Recovery using known transformations,  where the theoretical questions concern identifiability, recovery probability, error bounds, study of the loss landscape, etc, we are also interested in other directions of research:
\begin{enumerate}

    \item{Learning the transformations themselves.}

Study of families of transformations
    \[
    \{T_i:i \in I\}
    \]
    and study both the mathematical object and the transformation:
    \[
    (\widehat x,\widehat i)
    =
    \arg\min_{x,i}
    L(x,T_i).
    \]

    For mathematical databases, this becomes especially interesting because the learned transformation might correspond to an interpretable mathematical operation or rule.

    \item {Equation solving and conjecture generation.}

    Once transformations are known or learned, search the database for objects satisfying transformation equations such as $    F(T_1(x),\ldots,T_m(x))=0.$

    Approximate solutions can be ranked through
    \[
    \mathcal L(x)
    =
    d\!\left(
    F(T_1x,\ldots,T_mx),0
    \right),
    \]
for some metric d, 
    thus identifying objects that nearly satisfy a proposed mathematical relation.

    Examining these objects could suggest conjectures about which classes satisfy the equation exactly and why.
\end{enumerate}

\subsection{Loss functions and Axioms in the context of recovery from noisy data}

Let us consider a general axiomatic structure with a space of elements $r \in R$.  An Axiomatic Structure is presented as the objects of interest of study from Learning Theory (in this case Loss functions), along with a list of the axioms of the objects involved. The axioms can be divided into the ones characterizing the objects of interest $R$ and the ones of the transformations involved. The loss function $\mathcal{L}(\cdot)$ can have various formats involving the objects of study $r_1,\ldots, r_n,$ the involved parameters $s_1, \ldots, s_m, \lambda_1, \ldots, \lambda_m,$ the transformations $T_1, ...., T_m,$ and the losses used $ \ell_1(\cdot), ..., \ell_m(\cdot)$, for $n, m \in \mathbb{N}$. We focus this work on total loss functions in which the format is given by a sum of individual constraint components but one might consider more general loss functions.
For each term in the loss function we have the sign choices $s_j \in \{-1, +1\}$, $j \in {1,..., m}$ (which can be negative to accommodate the cases in which the loss function has terms with negative signs) and the corresponding transformations $T_j$, $j \in \{0,..., m\}$, where $m \geq 0$ is the number of elements in the loss function. Each term in the loss function has a penalty rate parameter $\lambda_j$, $j\leq m$.  These are all non-negative real scalar parameters. We take the view that the data $y_i(\omega)$ is noisy, for $\omega \in \Omega$, where $\Omega$ is a sample space.  These loss functions appear under the name of Empirical Loss in Statistical Learning theory and they are very important objects of study when one does not have access to the whole population of objects but to a sample of certain size. One can formulate also a population version for our objects. We focus this work on the empirical loss version:

\begin{align}
& \quad r \in R\nonumber\\
& \quad \mathcal{L}_{\lambda_!, \ldots, \lambda_m}(r, \omega)=s_0\frac{1}{n}\sum_{i=1}^n \ell_1(T_1({r}),{y_i(\omega)})+\ldots+s_m\lambda_m\ell_m\left(T_m (r)\right) \\
& \mathcal{L}_{\lambda_1, \dots, \lambda_m}^*=\min_{r \in R} \mathcal{L}_{\lambda_1, \ldots, \lambda_m}(r, \omega).
\end{align}

Throughout the paper, we call $\mathcal{L}(\cdot)$ the total losses 
A particular choice that one can study corresponds to $\ell(\cdot):=(\cdot)^2$ and $s_j=1$, for all $j \in \{0,..., m\}.$ Another possible choice is $\ell(\cdot):=|\cdot|$. One can consider other choices for $\ell(\cdot)$. 

For the Axioms description, let us consider $(R, \circ_1, \circ_2, \ldots, *_1, *_2, \ldots)$ with $*_i$ and $\circ_i$, $i \geq 1$,  satisfying certain axioms along with the transformations $T_j$ which respect certain axioms with respect to the operations:

For $i \in \{1, \ldots, n\}$ and $j \in \{1, \ldots, m\},$ we have
\begin{equation}
T_j(r_1*_ir_2)=\Phi_{i,j}(r_1, r_2, *_i), 
\end{equation}
where $\Phi_{i,j}()$ are functions of the generic elements $r_1, r_2 \in R$ and the operations.

Similarly, for $i \in \{1, \ldots, n\}$ and $j \in \{1, \ldots, m\},$ we have
\begin{equation}
T_j(r_1\circ_ir_2)=\tilde{\Phi}_{i,j}(r_1, r_2, \circ_i), 
\end{equation}
where $\tilde{\Phi}_{i,j}()$ are functions of the generic elements $r_1, r_2 \in R$ and the operations.

We consider the following three formulations.

Given the data points $y_i(\omega)$, $i \in \{1, \ldots, n\}$ and the transformations $T_j$, $j \in \{1, \ldots, n\}$, a loss $\ell_1(\cdot),$ let us consider 
\[
l(r, y)=\frac{1}{n}\sum_{i=1}^n \ell_1(T_1({r}),{y_i(\omega)})
\]
along with the constraints 
\[
C_{j}(T_j, r)
\]
where we simplified the notations, dropping the dependence on the transformations, for the convenience of the reader.
In the context of the Hard-constrained problems
one imposes the interpolation condition exactly
\[
\min_r l(r, y)
\qquad \text{subject to } C_j(T_j, r)=0.
\]

Then, one considers the Lagrangian
$$
\mathcal{L}(r,\lambda)=l(r, y)+\lambda^\top C(r),
$$
where \(\lambda\) is a Lagrange multiplier and the quantity $C(r)$ contains all the constraints we considered. 
Then, one considers the saddle point problem:
$$
\min_r \sup_{\lambda} \bigl(l(r,y)+\lambda^\top C(r)\bigr).
$$
This gives that for fixed \(r\), we have
$$
\sup_{\lambda}\bigl(l(r, y)+\lambda^\top C(r)\bigr)
=
\begin{cases}
l(r, y), & C(r)=0,\\[1mm]
+\infty, & C(r)\neq 0.
\end{cases}
$$

Thus infeasible values are excluded automatically, since if \(C(r)\neq 0\), the term
$\lambda^\top C(r)$
can become arbitrarily large.
Instead of imposing exact interpolation, one may only penalize deviation from it:
$$
\min_r \Bigl(l(r, y)+\lambda \|C(r)\|^2\Bigr).
$$
In this case the constraints \(C(r)\) are not forced. In the soft penalty approach, one balances the original objective \(l(r, y)\) against the cost of not respecting interpolation condition. In this case, $\lambda$ is not an optimization variable. It is a chosen parameter that one does not differentiate with respect to. As for fixed r, we have that $
l(r, y)+\lambda\|C(r)\|^2
$
is increasing in $\lambda$ whenever $C(r)\neq 0$, the minimizing would give $\lambda=0.$ In addition to the previous methods, one can consider a mix between the two, the so-called augmented Lagrangian. We refer the reader to \cite{veryimportantsoftandhardlagrangepaper2023augmented} for more details and an application of this method. In the cases that we study in this manuscript, we will be interested mostly in the soft penalty formulation, as we would like to detect from noisy observations objects that might not satisfy all of the constraints as well and to characterize how 'far ' are the objects from satisfying all of the constraints. The constraints appear in the models when one might have some access to some prior information.

\subsection{Gibbs measures and Axioms}

So far we have discussed loss-functions, now we move our attention to the energy formulation and its interaction with the axioms. Let us consider $(R, \circ_1, \circ_2, \ldots, *_1, *_2, \ldots)$ with $*_i$ and $\circ_i$, $i \geq 1$,  satisfying certain axioms. For a subset of elements $\mathcal{R} \subseteq R$, parameters $\lambda_j,$ losses $\ell_j(\cdot),$ $j \in \{1, \ldots, m\}$, we consider the following 'Energy linked with Axioms'  as objects of study of the 'Energy-based methods linked with Axioms':

\begin{align}
& \quad E_{\lambda_!, \ldots, \lambda_m}(r, \omega)=\lambda_1\ell_1\left(T_1 (r)\right)+\ldots+\lambda_m\ell_m\left(T_m (r)\right) \\
& E_{\lambda_1, \dots, \lambda_m}^*=\min_{r \in \mathcal{R}} E_{\lambda_1, \ldots, \lambda_m}(r, \omega).
\end{align}

These are presented along with a list axioms/rules for $T_j$, $j \in \{1, \ldots, m\}$ on the basic operations between the elements of the space.

For $\beta \geq 0$, whenever the  partition function $Z_\beta$  is finite, we consider the Gibbs measure
\begin{equation}
\mu_{\beta, \lambda_1, \ldots, \lambda_m}(r)
=
\frac{e^{-\beta E_{\lambda_1, \dots, \lambda_m}(r)}}{Z_{\beta}},
\end{equation}
where $Z_{\beta}$ acts as a normalization. 

Objects $r$ with very small energies are more likely to occur while objects with higher energies are less likely to appear under this measure. This is a general picture, which we particularly apply and detail in the following sections for transformations linked with certain frameworks. These Gibbs measures are useful, among others, when one is interested to do sampling of objects with certain properties in order to generate data.

\section{An $\text{ALS}_d$ linked with Arithmetic}


Let us consider in this section, the space of objects to learn from noisy data to be a subset of integer coefficients polynomials. We study an $ALS_d$ for these polynomials based on a mix of usual and p-derivations, where p is a prime number. 
A non-exhaustive list of references on p-derivations is \cite{Arithmeticdiffequationsbook}, \cite{Buiumoriginalpaper1}, \cite{Buiumoriginalpaper2}. We assume familiarity with the usual derivatives $\frac{d}{dx}$ of a polynomials, which satisfy Linearity and Leibniz rule, and we proceed with the introduction of p-derivations.

\subsection{Introduction to p-derivations}
Let $p$ be a prime number. From Fermat's Little Theorem, we have that for any integer $n$, 

$$
n \equiv n^p \quad \bmod p. 
$$
 
This result allows us to divide the difference $n-n^p$ by $p$. This will be useful to define the notion of p-derivation below.  

\begin{mydef}
For a prime number $p$, the $p$-derivation on $\mathbb{Z}$ is defined as the function $\delta_p: \mathbb{Z} \rightarrow \mathbb{Z}$ given by the formula

$$
\delta_p(n)=\frac{n-n^p}{p}.
$$
\end{mydef}

\begin{ex}
Let $n \in \mathbb{Z}$. The 19-derivation on $\mathbb{Z}$ is given  by
$$ \delta_{19}(n)=\frac{n-n^{19}}{19}.
$$
\end{ex}

This transformation is natural if one looks at the $p$-adic order of $n$, that is the number of copies of $p$ in its prime factorization. Let us write $n=p^a m$ with $\operatorname{gcd}(m, p)=1$, if $a>0$, we get

$$
\delta_p\left(p^a m\right)=\frac{p^a m-\left(p^a m\right)^p}{p}=p^{a-1} m\left(1-p^{a p-a} m^{p-1}\right).
$$

Since $p \geq 2$ and $a \geq 1$, we must have $a p-a \geq 1$, so $p$ does not divide $1-p^{a p-a} m^{p-1}$. Thus, we obtain that the $p$-derivation decreases the $p$-adic order of a multiple of $p$ by exactly one.

For $R$ a commutative ring with unit 1 and $p$ a prime integer, a p-derivation on $R$ is a function $\delta: R \rightarrow R$ such that $\delta(0)=\delta(1)=0$ and $\delta$ satisfies the Sum Rule/Axiom and the Product Rule/Axiom below; i.e., for all $r, s \in R$,

$$
\delta_p(r+s)=\delta_p(r)+\delta_p(s)-\sum_{i=1}^{p-1} \frac{\binom{p}{i}}{p} r^i s^{p-i},
$$

and

$$
\delta_p(r s)=r^p \delta_p(s)+s^p \delta_p(r)+p \delta_p(r) \delta_p(s).
$$

Thus, the axioms/rules above show that the $p-$derivations are not additive nor multiplicative making them different transformations from the usual derivatives.

We emphasize that the definition in \cite{Jeffriescite} covers the more general case of smooth $\mathbb{Z}$-algebras. For simplicity of the exposure, we will do our analysis in one of the simplest cases of this, that is the integer coefficients polynomials. The framework that we develop can be made more general. Note that a ring $R$ might have more derivations or p-derivations on it. In our work, we assume that we choose one of the derivations and one of the p-derivations that we use in the loss functions.

In the examples of $ALS$ which we consider in this manuscript, we work over the integer-coefficient polynomials. Given this, we will work with the following p-derivation in this context. 

\begin{mydef}\label{def p der}
On the ring of polynomials with integer coefficients in n variables $\mathbb{Z}[x_1, x_2, \ldots, x_N]$ we have
$$
\delta_{\text{st},p}\left(f\left(x_1, \ldots, x_N\right)\right)=\frac{f\left(x_1^p, \ldots, x_N^p\right)-f\left(x_1, \ldots, x_N\right)^p}{p},
$$
\end{mydef}

This is a $p$-derivation. Following the literature, we call this the standard p-derivation on $\mathbb{Z}\left[x_1, \ldots, x_N\right]$ and use the notation $\delta_{\text{st, }p}$ as in the definition.

In order to see how this p-derivation acts on polynomials with integer coefficients, we apply it concretely in the following example.

\begin{ex}
Let us consider the particular case $p=2$. Then for $f(x) \in \mathbb{Z}[x],$ $f(x)=x^2+3x,$ we have
$$
\begin{aligned}
\delta_{\mathrm{st}, 2}\left(x^2+3 x\right) & =\frac{\left(x^2\right)^2+3\left(x^2\right)-\left(x^2+3x\right)^2}{2} \\
& =-3x^2-3 x^3 .
\end{aligned}
$$

\end{ex}

In the closing of this section, we mention that not every ring admits a $p$-derivation. It is known that no ring $R$ of characteristic
$p$ admits a $p$-derivation, due to the following contradiction in $R$

$$
0=\delta(0)=\delta(p)=\delta_p(p)=1-p^{p-1}=1
$$
It is natural to ask which rings admit p-derivations.

\subsection{Polynomials in the set $A_N$}
Before we proceed with our analysis, let us mention the set of polynomials we will consider in our investigation. In our study we consider $f(x_1, \ldots, x_N) \in A^D_N \subseteq \mathbb{Z}[x_1, \ldots, x_N]. $ For some results that we cover, $A^d_N$ is set of all integer-coefficients polynomials of degree $D$ in several variables or a subset of it characterized by some conditions that the polynomials satisfy. 

Let $|d|=d_1+d_2+...+d_N$ and $f(x_1, \ldots, x_N)=\sum_{|d|\leq D}c_dx_1^{d_1}x_2^{d_2}...x_N^{d_N},$ be degree $D$ polynomials in several variables. We use the notation $deg(f)=D$ to denote the degree the polynomial.
One is interested also in sets of the form
$$A^D_N=\{f(x_1, \ldots, x_N)\in \mathbb{Z}[x_1, \ldots, x_N]||deg(f)=D, |c_d| \leq B_d\},$$
with $B_d$ some constants. One can indeed consider more general sets with different upper and lower bounds on the coefficients, or more general conditions on the coefficients. For convenience, we drop the index $D$, and from now on we use the notation $A_N$ for $A^D_N.$

\subsection{Symbolic powers of ideals and constraints}

First, we introduce the notion of the $L^{th}$ symbolic power of a prime ideal in a Noetherian ring.
\begin{mydef}
Let P be a prime ideal in a Noetherian ring R.
Then the $L^{th}$  symbolic power of P is 

\begin{equation}
P^{(L)}=\{r \in R | sr \in P^L \hspace{2mm}\text{for some s}  \notin P\},
\end{equation}
where $P^L$ is the $L^{th}$ power of the ideal $P$. 
\end{mydef}
We note that $P^L$ is generated by the L-fold products of the elements from P.  For more details on these algebraic structures, we refer the reader to \cite{Surverysymbolicpowers}.
We emphasize that this definition of the symbolic power of an ideal is one of Algebraic nature. 
The Zariski-Nagata Theorem gives a geometric interpretation of the symbolic power. In appropriate settings, symbolic powers can be characterized by higher-order vanishing along the variety.$$f\in P^{(L)} \quad\Longleftrightarrow\quad f \text{ vanishes to order at least }L \text{ along the variety  }V(P).$$ 

In characteristic zero, this expressed in terms of derivatives: $$\frac{d^{\alpha}f(x) }{dx^{\alpha}}|_{V(P)}=0 \qquad \text{for all } |\alpha|<L.$$ The characterization using the derivatives (in the complex case) and the derivatives and the p-derivations in the integer case it is a criterion that guarantees that the polynomial that respects those conditions is in a certain order of a symbolic power of an ideal. We note that these are 'interpolation' type conditions of higher order when viewed reduced mod p. When the ideal is of a certain form, say $q=(p, x-a)$ then the interpolation criterion is direct (see the pictures below). In higher dimensions when the ideal has as generator an irreducible polynomial then the condition is rephrased accordingly (see the discussion after the Theorem). In addition, this type of criterion is very useful as in general one only has $P^L \subset P^{(L)}$ but the reverse direction is not clear. Thus, having a criterion for the belonging to symbolic powers is very useful to check if the data generated using a true polynomial (to which we add noise) belongs to the symbolic power. Indeed, this criterion allows one to check for polynomials that might not be obtained directly through products of the generators of the ideal (that belong to the power of the ideal not directly to the symbolic power). If one learns that the polynomial belongs to a symbolic power of a certain ideal then given the algebraic characterization of this notion of symbolic power one can use this information and the properties of symbolic powers in various settings. For example, then one knows directly that multiplying the polynomial with another polynomial from another symbolic power of an ideal will result in a polynomial in the symbolic power of another order. In general, the symbolic powers characterize a certain order of contact on the variety given by the generators so this is also useful to know if this polynomial appears along other polynomials (random or deterministic) in certain systems of equations. This helps analyze the structure of the solution sets (see below). In general one only has $ P^2\subseteq P^{(2)}, $ and strict containment $ P^2\subsetneq P^{(2)}  $ can occur. In \cite{Jeffriescite}, the authors establish a  characterization of elements of symbolic powers using mixed normal and  p-derivations.

In Machine Learning one searches for minimizers of a certain loss function. Also in this area, problems concerning interpolation of noisy data appear naturally (see \cite{murphy2012machine}. For a collection of data points $y_i (\omega)$, let us consider integer coefficients polynomials
$f(x) \in \mathbb{Z}[x].$ Let nonnegative $\lambda \in \mathbb{R_{+}},$ be a scalar modeling the penalty rate. For these objects, we consider the $ALS_d$ with the following loss function

\begin{align*}
&\quad \mathcal{L}_{\lambda}(f, \omega)=\frac{1}{n}\sum_{i=1}^n (T_1({f}\left(x_i \right))-{y_i(\omega)})^2+\lambda \left(T_2 f\right)^2+ \ldots+ \lambda \left(T_m f\right)^2\\
& 
\mathcal{L}_{\lambda}^*=\min_{f \in \mathbb{Z}[x] \cap A_1} \mathcal{L}_{\lambda}(f, \omega).
\end{align*}

We emphasize that although we have already used the letter $n$ when we introduced p-derivations, we will abuse the notation and use it also for the number of data-points.
One can consider more involved loss functions which have more terms in their structure. These transformations are either given by their definition, or by the rules they satisfy when considered on elements on the space and the operations between them. In parts of our analysis, we will consider $y_i(\omega)$ to be outputted by a true-polynomial $f_*(x) \in \mathbb{Z}[x]$ perturbed by noise.  We perform some parts of the analysis in these cases but the setting is general in the sense that $y_i(\omega)$ might come from general real-data not necessarily synthetic.

In the following in the one-variable case, we use the notation $f\hspace{1mm}  \text{mod} q$, for the remainder of the division between the $f$ polynomial and the polynomial in the generator of the ideal $q$. As an example, in the one-dimensional case we work with ideals of the form $q=(p, x-a)$. In the following, we consider $l \in \mathbb{N}$.

\begin{itemize}
\item[]1) A constraint of the form $T_2(f):= \frac{\partial^{b_2}}{\partial x^{b_2}} f \text{mod} \hspace{1mm} q$, for integer $b_2 \geq 1$, for some ideal of the form $q=(x-a)$, $a \in \mathbb{Z}$ is a \textit{interpolation with a given multiplicity/contact} type condition for the polynomial.

\item[] 2) A constraint of the form $T_2(f):= \delta_{s t, p}f \text{mod} \hspace{1mm} q$, for some ideal of the form $q=(p,x-a)$, $a \in \mathbb{Z}$ is one of an \textit{arithmetic type}.

\item[] 3) Constraints of the form $\delta^{b_1}_{s t, p}(\frac{\partial^{b_2}}{\partial x^{b_2}} f) \text{mod} \hspace{1mm} q$, for $b_1+b_2 \leq L-1$, for integer $L \geq 1$, for some ideal of the form $q=(p, x-a), $ $a \in \mathbb{Z}$, are of a \textit{mixed type} which capture \textit{an interpolation with a given multiplicity along with arithmetic conditions}.

\end{itemize}

One can consider the case in which the ideal does not contain the prime number p. Let us consider the ALS in which the  transformation $T_2(\cdot)$ involves the standard derivatives. Then, one obtains loss functions with derivatives of higher orders. 
In addition, one can consider other examples of constraints (see \cite{BishopML}, p.164) for example of constraints which control the size of the coefficients.

\subsection{One-variable case}

We first consider the following case. 
Let $p$ be a fixed prime number. Let $M_p$ be a set of ideals that contain a prime number of the form $(p, x-a),$ for $ a \in \mathbb{Z}$, or in a subset of it. For simplicity, in order to also match existing literature (\cite{BishopML}), we consider the least-squares residuals. The real parameter $\lambda$ is called the penalty rate and it penalizes errors from the part of the loss function that measures the order of the roots of the polynomial.

 One model we consider consists of noisy observations $y_1(\omega)$, $y_2(\omega),\ldots, y_n(\omega)$ obtained as a noisy perturbations of an integer-coefficient polynomial.
That is, we consider a model for our data in which they are generated using a deterministic polynomial with noise (it can be Gaussian noise or other type of noise). Then, one wants to study the landscape of the loss functions for these data sets.  Optimization problems in which one has interpolation conditions in general or for the standard derivative or gradient as constraints appear in the literature (see \cite{veryimportantsoftandhardlagrangepaper2023augmented}, \cite{booksurfacereconstructionpage5322015scale}, \cite{CSimagereconstructpaperrobust}).

Let integers $b_1,b_2$ such that $b_1+b_2 \leq L-1$, for some integer $L\geq 1$. The parameter $L \in \{1, 2, \ldots, L_{max}\}$, $L_{max} \in \mathbb{N}$, controls the corresponding level of the symbolic power of the ideal considered (see (\cite{Jeffriescite}). For simplicity of the exposure, we consider the same parameter $\lambda$ for all the constraint terms, but one can consider a more general version in which each constraint has its own parameter $\lambda_j,$ $j \in \{1, 2, ..., m\}.$ Let us consider the $ALS_d$ with the following sequence of loss functions indexed by parameters along with the corresponding axioms of the elements involved.

\begin{align*}
& \quad \mathcal{L}_{L, \lambda, q}(f, \omega)=\frac{1}{n}\sum_{i=1}^n \ell({f}\left(x_i\right)-{y_i(\omega)})+\lambda \sum_{b_1+b_2 \leq L-1}\left(\delta^{b_1}_{s t, p}\left(\frac{\partial^{b_2}}{\partial x^{b_2}} f\right) \text{mod} \hspace{1mm} q\right)^2 \\
& \mathcal{L}_{L, \lambda, M_p}^*=\min_{ q \in M_p}\min_{f \in \mathbb{Z}[x] \cap A_1} \mathcal{L}_{L, \lambda, q}(f, \omega).
\end{align*}
 Note that the expression in the regularization terms can be understood as follows. After the operations are applied, we obtain from the polynomial $f(x)$, a new polynomial, then we consider the remainder of that polynomial with respect to the polynomial $x-a$, which is an integer (degree of the remainder is zero in this context). We look at the closest multiple of $p$ close to that integer. We use the convention, that for integer numbers $n$ we consider $d_p(n):=
\min_{k\in\mathbb Z}|n-kp|.$ Thus, we have $|h\bmod q|:=d_p(h(a))$.Then, the squared gives a positive contribution to the loss function whenever the polynomial does not pass exactly through a multiple of p.
 
For the least-squares case, we obtain
\begin{equation}
\mathcal{L}_{L, \lambda, q}(f, \omega)=\frac{1}{n}\sum_{i=1}^n({f}\left(x_i \right)-{y_i(\omega)})^2+\lambda \sum_{b_1+b_2 \leq L-1}\left(\delta^{b_1}_{s t, p}\left(\frac{\partial^{b_2}}{\partial x^{b_2}} f\right) \text{mod} \hspace{1mm} q\right)^2,
\end{equation}

where $\delta_{\text{st},p}: \mathbb{Z}[x] \to \mathbb{Z}[x]$ is the standard $p$-derivation 

$$
\delta_{\text{st},p}(f)=\frac{f\left(x^p\right)-f\left(x\right)^p}{p}.
$$

In general, for p-derivations on rings we have the following.
A set-theoretic map $\delta: R \rightarrow R$ is a $p$-derivation if it satisfies the following Rules/Axioms. 

\begin{itemize}

\item[] a) Axiom/Rule 1: $\delta(1)=0$ 
\end{itemize}

$\delta$ satisfies the following identities for all $x, y \in R$ :
\begin{itemize}
\item[] b) Axiom/Rule 2: $$
\delta(x y)=x^p \delta(y)+y^p \delta(x)+p \delta(x) \delta(y)
$$
\item[] c) Axiom/Rule 3: $$
\delta(x+y)=\delta(x)+\delta(y)+\mathcal{C}_p(x, y),
$$
where $\mathcal{C}_p(X, Y)=\frac{X^p+Y^p-(X+Y)^p}{p} \in \mathbb{Z}[X, Y]$.

\end{itemize}

We emphasize that one can work directly with definitions or with the defining axioms that characterize these transformations.  We work with the axioms as they are important from the computations viewpoint when we explore the landscape of the loss functions. 

Depending on the problem that one studies one can add more properties of these $p$-derivations to the $ALS_d$. For example, from the Axiom/Rule 2 along with induction, one can obtain an analogue of the power rule for these transformation:
\begin{itemize}
    \item[] Power Rule Analogue: $$
\delta_p\left(n^a\right)=\sum_{i=1}^a\binom{a}{i} p^{i-1} \delta_p(n)^i n^{(a-i) p}.
$$
\end{itemize}

In addition, one can also deduce other properties of these transformations depending on the operations between elements that appear naturally in the problem considered from the Axioms/Rules 1), 2) and 3).

We work with these structures based on axioms as from these rules one can generate new properties which might be needed in the problems under consideration. We take the view that a p-derivation is a transformation that satisfies the Axioms 1, 2, and 3 and study the corresponding loss functions involving these transformations.

We consider a model in which the set $M_p$ where we search for these ideals can be considered is the same for all the realizations, We consider the set of deterministic ideals that we search over $M_p=\{q_1, q_2,\ldots, q_s\}$ for some $s\geq 1$, and ideals of the form $q=(p, x-a_i)$, for $i \in \{1, \ldots, s\}$ in order to find if the potential ground-truth polynomial which generated the data after perturbation is a member of any symbolic power of an ideal among the search space, or close to being a member. This version is deterministic and it might appear in various contexts, including the ones in which you see the potentially recovered polynomial as being part of a system of polynomials for which one knows already that there are roots at the integer points which describe (part of) the ideals. Thus, one wants to learn of the potential ground truth polynomial also has roots there. 

In a later section, we do a random search version for these types of problems.

\begin{rmk}[Motivation to consider the soft version]

 The soft penalty version can be better in practice if no strong prior information is known. For example, one might not know a-priori, how many order of derivatives and p-derivations to take. Moreover, when considering the soft formulation of the problem, one can also study the values of the loss functions on polynomials which respect a part but not all of the conditions. Thus, one also sees which conditions fail for the corresponding polynomial to be in the symbolic power of the ideal. Moreover, simultaneous ideal membership across multiple anchors is generically infeasible. The hard constraint then forces the estimator into a degenerate corner.
\end{rmk}

\begin{figure}[H]
        \centering
        \includegraphics[width=.45\linewidth, trim=1.5cm 1cm 1.5cm 1cm, clip]{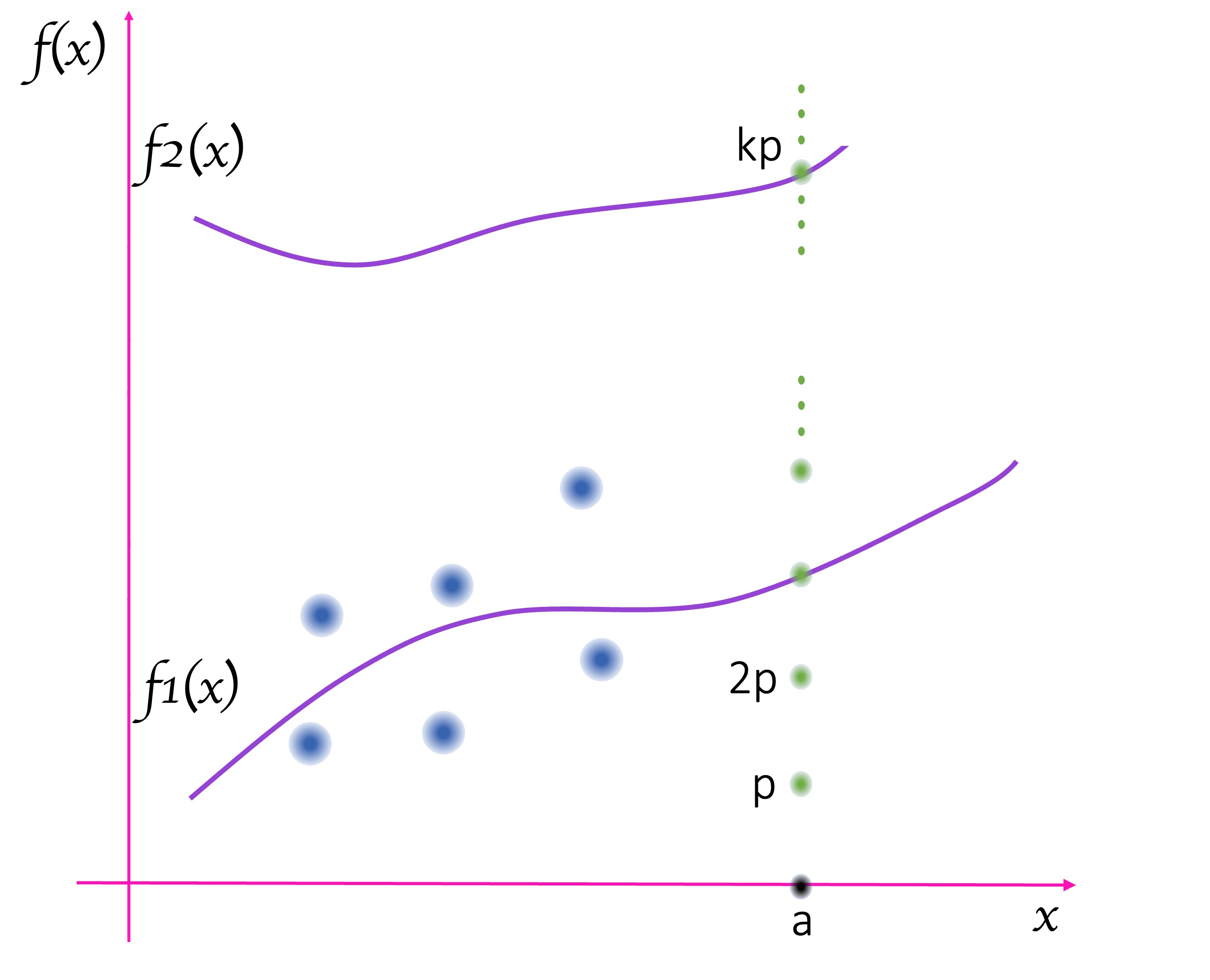}
        \caption{For $q=(p,x-a)$ for some $a \in \mathbb{Z}$, this is a schematic representation of an integer-coefficient polynomials that give zero contribution in a part of the penalty term of a loss function for which $q \in M_p,$ but which give different least-squares contribution as $f_1(x)$ is closer to the data (represented as blue dots) than $f_2(x).$}
\end{figure}

\begin{figure}[H]
        \centering
        \includegraphics[width=.45\linewidth, trim=1.5cm 1cm 1.5cm 1cm, clip]{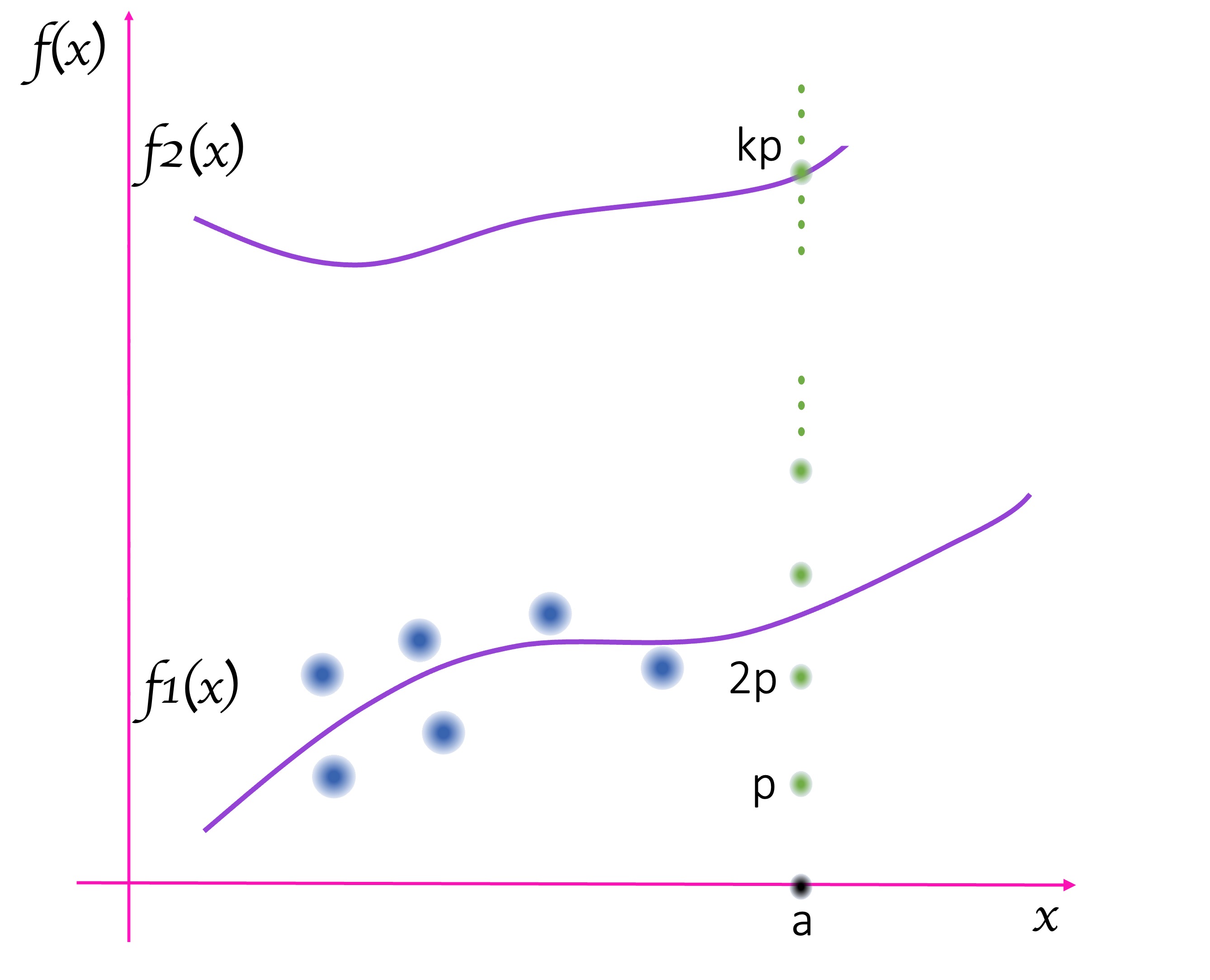}
        \caption{For $q=(p,x-a)$ for some $a \in \mathbb{Z}$, this is a schematic representation of an integer-coefficient polynomial that gives non-zero contribution in a part of the penalty term of a loss function for which $q \in M_p$.}
\end{figure}

\subsubsection{Other variants of $\text{ALS}_d$ for polynomials with multiple roots}

In addition, one can search for polynomials with multiple roots of certain multiplicity and design a loss function with this purpose. For this, let us consider the prime ideals $q_j=(p, x-a_j)$, for $a_j \in \mathbb{Z}.$

Let $J \in \mathbb{N}$ and let  $\tilde{q}_J:=\cap_{j=1}^Jq_j$. Then, one to recover polynomials  with roots of certain multiplicities from noisy data, one is interested in an $ALS_d$ with the same objects along with the following sequence of loss functions. 

\begin{align*}
& \quad \mathcal{L}_{L, \lambda, \tilde{q}_J}(f, \omega)=\frac{1}{n}\sum_{i=1}^n ({f}\left(x_i\right)-{y_i(\omega)})^2+\lambda \sum_{b_1+b_2 \leq L-1} \left(\delta^{b_1}_{s t, p}\left(\frac{\partial^{b_2}}{\partial x^{b_2}} f\right) \text{mod} \hspace{1mm} \tilde{q}_J\right)^2 \\
& \mathcal{L}_{L, \lambda, J}^*=\min_{\tilde{q}_j}\min_{f \in \mathbb{Z}[x] \cap A_1} \mathcal{L}_{L, \lambda, \tilde{q}_J}(f, \omega).
\end{align*}

We note that we consider the same multiplicity $L \geq 1$ for the root in this notion of the loss function, but one can consider a more general version with various multiplicities of roots. 
\begin{align*}
& \quad \mathcal{L}_{L
_J, \lambda, \tilde{q}_J}(f, \omega)=\frac{1}{n}\sum_{i=1}^n ({f}\left(x_i\right)-{y_i(\omega)})^2+\lambda \sum_{b_1+b_2 \leq L_j-1} \left(\delta^{b_1}_{s t, p}\left(\frac{\partial^{b_2}}{\partial x^{b_2}} f\right) \text{mod} \hspace{1mm} \tilde{q}_J\right)^2 \\
& \mathcal{L}_{L_J, \lambda, J}^*=\min_{\tilde{q}_j}\min_{f \in \mathbb{Z}[x] \cap A_1} \mathcal{L}_{L_J, \lambda, \tilde{q}_J}(f, \omega).
\end{align*}

We note that one can consider versions of these loss functions for more prime numbers.

\subsection{Multi-variables case}

Let p be a prime number. Let $M_p$ be a set of ideals that contain a prime number of the form $(p, x_1-a_1, x_2-a_2, \ldots, x_N-a_N),$ for $ a_i \in \mathbb{Z}$, $i \in \{1, 2, \ldots, N\}$ or in a subset of it. The remaining parameters are the same as in the one-dimensional case.  
Let us consider $J \in \mathbb{N}, J \geq 1$, and the indices $j_1+j_2+...+j_N<J.$ 

Let  $g \in \mathbb{Z}\left[x_1, \ldots, x_N \right]$. For these objects, one can consider the following multi-dimensional sequence of loss functions 

\begin{align}
\mathcal{L}_{L, \lambda, q}(g, \omega)&=\frac{1}{n}\sum_{i=1}^n
({g}\left(x_1,  \ldots, x_N\right)-{y_i(\omega)}
)^2\nonumber\\
&+\lambda \sum_{b_1+j_1+\ldots+j_N\leq L-1}||\delta^{b_1}_{s t, p}\left(\frac{\partial^{j_1+...j_N}}{\partial x_1^{j_1} \ldots \partial x_N^{j_N}} g\left(x_1, \ldots, x_N\right)\right) \text{mod} \hspace{1mm} q||^2
\end{align}
\begin{equation}
\mathcal{L}_{}^*=\min_{q \in M_p}\min_{g \in \mathbb{Z}\left[x_1, \ldots, x_N\right] \cap A_N}\mathcal{L}(g, \omega). \\
\end{equation}

\begin{rmk}
The discussion in the one-dimensional case is also valid in the multi-dimensional case. Indeed, one can consider the $ALS_d$ of multi-variables polynomials with the standard derivative.
\end{rmk}



The following result is also motivated by the fact that in this setting there is no continuous gradient method as the dataset is discrete. One is interested to study the loss landscape. One aim of this result is to show how the loss functions landscape is influenced by the axioms of the operations involved-in this case, usual derivatives and p-derivations. In the case of integer coefficients polynomials, addition between the elements is a way to explore the space. In other spaces such as graphs or lattices other operations might be natural when studying the Loss functions landscape. For example, one might study other data-sets such as the ones formed by polynomials with coefficients of the form $a +bh$ where $h$ is some integer step and $a, b \in \mathbb{Z}.$

In the theorem below we focus on a certain class of prime ideals (see the remark after the Theorem) and on the derivatives and p-derivations but a similar statement can be obtained using some other transformations given by other axioms. To simplify the notation, we drop the index $i$ in $(x^i_1, \ldots, x^i_N)$, $i\leq n$, for each evaluation point and we simply write $(x_1, \dots, x_N).$

\begin{thrm}\label{result1}
Let $q$ be a prime ideal from a certain class of $\mathbb{Z}[x_1, \ldots, x_N]$ containing the prime number p. For each realization of the noisy data (specific sample dataset) $(y_1(\omega), \ldots, y_n(\omega))$ we have the following.  \\
1)Let $f_0(x_1,\ldots, x_N)$ be an integer-coefficients polynomial. Let us consider the transformation $$T_1^r:=f_0(x_1,\ldots, x_N)+r(x_1,\ldots, x_N).$$  Then, we have that
\begin{align}
&\mathcal{L}_{L, \lambda, q}(f_0+r, \omega)=\frac{1}{n}\sum_{i=1}^n
({f_0}\left(x_1,  \ldots, x_N \right)-{y_i(\omega)}
)^2+2r(x_1, \ldots x_N)(f_0(x_1, \ldots, x_N)-y_i(\omega))+(r(x_1, \ldots, x_N))^2\nonumber\\
&+\lambda \sum_{j_1+\ldots+j_N\leq L-2}||\left(\delta_{s t, p}\frac{\partial^{j_1+...j_N}}{\partial x_1^{j_1} \ldots \partial x_N^{j_N}} f_0\left(x_1, \ldots, x_N\right)+ \delta_{s t, p}\frac{\partial^{j_1+...j_N}}{\partial x_1^{j_1} \ldots \partial x_N^{j_N}}r(x_1, \ldots, x_N)\right)\text{mod} \hspace{1mm} q\nonumber\\
&-\sum_{i=1}^{p-1} \frac{\binom{p}{i}}{p} (\frac{\partial^{j_1+...j_N}}{\partial x_1^{j_1} \ldots \partial x_N^{j_N}}f_0(x_1, \ldots, x_N))^i (\frac{\partial^{j_1+...j_N}}{\partial x_1^{j_1} \ldots \partial x_N^{j_N}}r(x_1, \ldots, x_N))^{p-i} \text{mod} \hspace{1mm} q||^2+\mathcal{E}(b_1 \neq 1),
\end{align}
where we used the notation $\mathcal{E}(b_1\neq 1)$ for the extra terms in the expansion.

2)Let $\tau \geq 2$. Let us consider $f(x_1,\ldots, x_N)$ be an integer coefficient polynomial as before. Let us consider another transformation $\tilde{T}_\tau^r(f)$ with 
\begin{equation}
\tilde{T}_\tau^r(f):=f_0(x_1,\ldots, x_N)+r^\tau(x_1,\ldots, x_N).
\end{equation}
Then, we have that
\begin{align}
&\mathcal{L}_{L, \lambda, q}(f_0+r^\tau, \omega)=\frac{1}{n}\sum_{i=1}^n
({f_0}\left(x_1,  \ldots, x_N\right)-{y_i(\omega)}
)^2+2r^\tau(x_1, \ldots x_N)(f_0(x_1, \ldots, x_N)-y_i(\omega))+(r^\tau(x_1, \ldots, x_N))^2\nonumber\\
&+\lambda \sum_{j_1+\ldots+j_N\leq L-2}||\left(\delta_{s t, p}\frac{\partial^{j_1+...j_N}}{\partial x_1^{j_1} \ldots \partial x_N^{j_N}} f_0\left(x_1, \ldots, x_N\right)+ \delta_{s t, p}\frac{\partial^{j_1+...j_N}}{\partial x_1^{j_1} \ldots \partial x_N^{j_N}}r^\tau(x_1, \ldots, x_N)\right)\text{mod} \hspace{1mm} q\nonumber\\
&-\sum_{i=1}^{p-1} \frac{\binom{p}{i}}{p} (\frac{\partial^{j_1+...j_N}}{\partial x_1^{j_1} \ldots \partial x_N^{j_N}}f_0(x_1, \ldots, x_N))^i (\frac{\partial^{j_1+...j_N}}{\partial x_1^{j_1} \ldots \partial x_N^{j_N}}r^\tau(x_1, \ldots, x_N))^{p-i} \text{mod} \hspace{1mm} q||^2+\mathcal{E}(b_1 \neq 1).
\end{align}

\end{thrm}

\begin{rmk}
We restrict to cases in which the prime ideals are of a certain format such that one can compare the resulting polynomial after the relations obtained from the ideal generators were considered.
We emphasize that the notation $(\cdot) \mod q$ means that we compare the resulting polynomial with the zero polynomial after reducing mod p. For example, if $q=(p, x-a) $ in one dimension then we compute the remainder of the resulting polynomial mod p at the point $x=a$. For more variables, for example, if $h(x_1, x_2)=x_1-x_2, $ then we compare the resulting polynomial (after the corresponding transformations involving standard and p-derivations) with the 0- polynomial mod p after we make the identification  $x_1=x_2$ in its variables.
For the resulting polynomial, if not identically zero, we compute the norm using the corresponding values of the coefficients. We use the convention, that for each integer coefficient $c_i$ we use $d_p(c_i)=\min_{k\in\mathbb Z}|c_i-kp|,$ when computing the norm. Two different polynomials might receive the same penalty but their data-fit terms might be different. In addition, one can add an extra squared-norm of the polynomial penalty term in the loss function if one is interested in polynomials with small coefficients in their models. In general, one can consider more refined versions of this analysis.
\end{rmk}

\begin{proof}

Let us consider a fixed polynomial $f_0(x_1, \ldots, x_N) \in A_N \subseteq \mathbb{Z}[x_1, \ldots, x_N], $ and let us consider explorations of the landscape from this fixed polynomial. First we study this type of explorations $f(x_1,\ldots, x_N)+r(x_1, \ldots, x_N)$. 
Then, 
\begin{align}
&\mathcal{L}_{L, \lambda, q}(f_0+r, \omega)=\frac{1}{n}\sum_{i=1}^n
({f_0}\left(x_1,  \ldots, x_N \right)+r(x_1, \ldots x_N)-{y_i(\omega)}
)^2\nonumber\\
&+\lambda \sum_{b_1+j_1+\ldots+j_N\leq L-1}||\delta^{b_1}_{s t, p}\left(\frac{\partial^{j_1+...j_N}}{\partial x_1^{j_1} \ldots \partial x_N^{j_N}} f_0\left(x_1, \ldots, x_N\right)+\frac{\partial^{j_1+...j_N}}{\partial x_1^{j_1} \ldots \partial x_N^{j_N}}r(x_1, \ldots, x_N)\right) \text{mod} \hspace{1mm} q||^2\nonumber\\
&=\frac{1}{n}\sum_{i=1}^n
({f_0}\left(x_1,  \ldots, x_N\right)-{y_i(\omega)}
)^2+2r(x_1, \ldots x_N)(f_0(x_1, \ldots, x_N)-y_i(\omega))+r(x_1, \ldots, x_N)^2\nonumber\\
&+\lambda \sum_{j_1+\ldots+j_N\leq L-2}||\delta_{s t, p}\left(\frac{\partial^{j_1+...j_N}}{\partial x_1^{j_1} \ldots \partial x_N^{j_N}} f_0\left(x_1, \ldots, x_N\right)+\frac{\partial^{j_1+...j_N}}{\partial x_1^{j_1} \ldots \partial x_N^{j_N}}r(x_1, \ldots, x_N)\right) \text{mod} \hspace{1mm} q||^2\nonumber\\
&+\lambda\sum_{j_1+\ldots+j_N\leq L-3}||\delta^2_{s t, p}\left(\frac{\partial^{j_1+...j_N}}{\partial x_1^{j_1} \ldots \partial x_N^{j_N}} f_0\left(x_1, \ldots, x_N\right)+\frac{\partial^{j_1+...j_N}}{\partial x_1^{j_1} \ldots \partial x_N^{j_N}}r(x_1, \ldots, x_N)\right) \text{mod} \hspace{1mm} q||^2+\mathcal{E}(b_1 \geq 3)\nonumber\\
&=\frac{1}{n}\sum_{i=1}^n
({f_0}\left(x_1,  \ldots, x_N \right)-{y_i(\omega)}
)^2+2r(x_1, \ldots x_N)(f_0(x_1, \ldots, x_N)-y_i(\omega))+r(x_1, \ldots, x_N)^2\nonumber\\
&+\lambda \sum_{j_1+\ldots+j_N\leq L-2}||\left(\delta_{s t, p}\frac{\partial^{j_1+...j_N}}{\partial x_1^{j_1} \ldots \partial x_N^{j_N}} f_0\left(x_1, \ldots, x_N\right)+ \delta_{s t, p}\frac{\partial^{j_1+...j_N}}{\partial x_1^{j_1} \ldots \partial x_N^{j_N}}r(x_1, \ldots, x_N)\right)\text{mod} \hspace{1mm} q\nonumber\\
&-\sum_{i=1}^{p-1} \frac{\binom{p}{i}}{p} (\frac{\partial^{j_1+...j_N}}{\partial x_1^{j_1} \ldots \partial x_N^{j_N}}f_0(x_1, \ldots, x_N))^i (\frac{\partial^{j_1+...j_N}}{\partial x_1^{j_1} \ldots \partial x_N^{j_N}}r(x_1, \ldots, x_N))^{p-i} \text{mod} \hspace{1mm} q||^2+\mathcal{E}(b_1 \neq 1).
\end{align}
We used the notation $\mathcal{E}(b_1 \neq 1)$ for the extra terms in the expansion.

 Let $\tau \geq 1.$ For the power-steps moves, we need the multiplication Axiom for both normal derivatives and p-derivations, as well.

\begin{align}
&\mathcal{L}_{L, \lambda, q}(f_0+r^\tau, \omega)=\frac{1}{l}\sum_{i=1}^l
({f_0}\left(x_1,  \ldots, x_N \right)+r^{\tau}(x_1, \ldots x_N)-{y_i(\omega)}
)^2\nonumber\\
&+\lambda \sum_{b_1+j_1+\ldots+j_N\leq L-1}||\delta^{b_1}_{s t, p}\left(\frac{\partial^{j_1+...j_N}}{\partial x_1^{j_1} \ldots \partial x_N^{j_N}} f_0\left(x_1, \ldots, x_N\right)+\frac{\partial^{j_1+...j_N}}{\partial x_1^{j_1} \ldots \partial x_N^{j_N}}r^{\tau}(x_1, \ldots, x_N)\right) \text{mod} \hspace{1mm} q||^2\nonumber\\
\end{align}
We further obtain
\begin{align}
&\frac{1}{n}\sum_{i=1}^n
({f_0}\left(x_1,  \ldots, x_N \right)-{y_i(\omega)}
)^2+2r^\tau(x_1, \ldots x_N)(f_0(x_1, \ldots, x_N)-y_i(\omega))+(r^\tau(x_1, \ldots, x_N))^2\nonumber\\
&+\lambda \sum_{j_1+\ldots+j_N\leq L-2}||\delta_{s t, p}\left(\frac{\partial^{j_1+...j_N}}{\partial x_1^{j_1} \ldots \partial x_N^{j_N}} f_0\left(x_1, \ldots, x_N\right)+\frac{\partial^{j_1+...j_N}}{\partial x_1^{j_1} \ldots \partial x_N^{j_N}}r^\tau(x_1, \ldots, x_N)\right) \text{mod} \hspace{1mm} q||^2\nonumber\\
&+\lambda\sum_{j_1+\ldots+j_N\leq L-3}||\delta^2_{s t, p}\left(\frac{\partial^{j_1+...j_N}}{\partial x_1^{j_1} \ldots \partial x_N^{j_N}} f_0\left(x_1, \ldots, x_N\right)+\frac{\partial^{j_1+...j_N}}{\partial x_1^{j_1} \ldots \partial x_N^{j_N}}r^\tau(x_1, \ldots, x_N)\right) \text{mod} \hspace{1mm} q||^2+\mathcal{E}(b_1 \geq 3)\nonumber\\
\end{align}
\begin{align}
&=\frac{1}{n}\sum_{i=1}^n
({f_0}\left(x_1,  \ldots, x_N \right)-{y_i(\omega)}
)^2+2r^\tau(x_1, \ldots x_N)(f_0(x_1, \ldots, x_N)-y_i(\omega))+(r^\tau(x_1, \ldots, x_N))^2\nonumber\\
&+\lambda \sum_{j_1+\ldots+j_N\leq L-2}||\left(\delta_{s t, p}\frac{\partial^{j_1+...j_N}}{\partial x_1^{j_1} \ldots \partial x_N^{j_N}} f_0\left(x_1, \ldots, x_N\right)+ \delta_{s t, p}\frac{\partial^{j_1+...j_N}}{\partial x_1^{j_1} \ldots \partial x_N^{j_N}}r^\tau(x_1, \ldots, x_N)\right)\text{mod} \hspace{1mm} q\nonumber\\
&-\sum_{i=1}^{p-1} \frac{\binom{p}{i}}{p} (\frac{\partial^{j_1+...j_N}}{\partial x_1^{j_1} \ldots \partial x_N^{j_N}}f_0(x_1, \ldots, x_N))^i (\frac{\partial^{j_1+...j_N}}{\partial x_1^{j_1} \ldots \partial x_N^{j_N}}r^\tau(x_1, \ldots, x_N))^{p-i} \text{mod} \hspace{1mm} q||^2+\mathcal{E}(b_1 \neq 1).
\end{align}

\end{proof}

The result also gives some information on the complexity of the landscape and the corresponding challenges to study the local minima in the corresponding frameworks. In addition, given the positive and negative signs of the coefficients in the $r(\cdot)$  polynomial along with the rules of the transformations, one obtains negative and positive terms, thus there are cancellations that one can study in these problems. 

\begin{rmk}
In order to make the structure of the terms appearing more apparent, we consider the particular case of the one-variable polynomials. Let us consider a simple example with $\lambda=1$ to ease the notation.

We have

\begin{align}
&\mathcal{L}_{3, \lambda, q}(f_0(x)+r^\tau(x), \omega)=\frac{1}{n}\sum_{i=1}^n
({f_0}\left(x\right)-{y_i(\omega)}
)^2+2r^\tau(x)(f_0(x)-y_i(\omega))+(r^\tau(x))^2\nonumber\\
&+((f_0(x)+r^\tau(x))\text{mod}q)^2+((\frac{\partial}{\partial x}(f_0(x)+r^\tau(x)))\text{mod}q)^2+((\delta_{st, p}(f_0(x)+r^\tau(x)))\text{mod}q)^2\nonumber\\
&+((\frac{\partial^2}{\partial x^2}(f_0(x)+r^\tau(x)))\text{mod}q)^2+
(\delta_{st, p}(\frac{\partial}{\partial x}(f_0(x)+r^\tau(x)))\text{mod}q)^2+
(\delta^2_{st, p}(f_0(x)+r^\tau(x))
\end{align}
The usual derivative satisfies Linearity, Leibniz Rule, and one has also the chain rule that we use along with the axioms of the p-derivations to obtain

\begin{align}
&((f_0(x)+r^\tau(x))\text{mod}q)^2+((\frac{\partial}{\partial x}f_0(x)+\tau r^{\tau-1}(x)(\frac{\partial}{\partial x}r(x))\text{mod}q)^2\nonumber\\
&+(\delta_{st, p}f_0(x)+ \delta_{st, p}(r^\tau(x)))\text{mod}q)^2-\sum_{i=1}^{p-1} \frac{\binom{p}{i}}{p}(f_0(x))^i(r^\tau(x))^{p-i})\text{mod}q)^2+\mathcal{E}\nonumber\\
&=((f_0(x)+r^\tau(x))\text{mod}q)^2+((\frac{\partial}{\partial x}f_0(x)+\tau r^{\tau-1}(x)(\frac{\partial}{\partial x}r(x))\text{mod}q)^2\nonumber\\
&+(\delta_{st, p}f_0(x)+ r^p\delta_{st, p}(r^{\tau-1}(x))+ (r^{\tau-1})^p\delta_{st, p}r(x)+ p\delta_{st, p}(r^{\tau-1}\delta_{st, p}(r)\text{mod}q)^2- \sum_{i=1}^{p-1} \frac{\binom{p}{i}}{p}(f_0(x))^i(r^\tau(x))^{p-i})\text{mod}q)^2+\mathcal{E}\nonumber\\
\end{align}
where $\mathcal{E}$ is a notation for the extra terms of higher order derivatives, p-derivations, and mixed ones. 
Then, for general $L\geq 3$ one adds more terms to the loss function.

\end{rmk}

\begin{rmk}[Study of other quantities from the axioms]
As a second remark we discuss how the axioms can be used to generate other quantities to be studied. 
Let $T_{n}^r(f)=f(x_1,\ldots, x_N)+r_1(x_1,\ldots, x_N)+ \ldots +r_{n}(x_1, \ldots, x_N).$
We recall that we take the view that we study the axiomatic structures as they give us a framework to recover other properties of the transformations from their axioms. For example, for $b_1=1$ , $b_2=0$, in order to compute $$\delta_{st, p}(T_n^r(f))=\delta_{st, p}(T_{n-1}^r(f))+\delta_{st, p}(r_N(x_1, \ldots, x_N))-\sum_{i=1}^{p-1} \frac{\binom{p}{i}}{p} (T_{n-1}^r(f))^i (r_N(x_1\ldots,x_N))^{p-i}, $$
with $T_{n-1}^r(f)=f(x_1,\ldots, x_N)+r_1(x_1,\ldots, x_N)+ \ldots +r_{n-1}(x_1, \ldots, x_N).$  We can then apply sequentially the axioms for $T_{N-1}$ and so on, and obtain a formula for the transformation $\delta_{st, p}(T^r_n(f))$ (which we can see as a property) via the Sum Rule/Axiom and expansions. We can do a similar analysis for any choice of $b_1, b_2$ with $b_1+b_2 \leq L-1. $ One can take this view in general for other transformations as well that can be reduced to the axioms on basic operations that helps us explore the data. 

\end{rmk}


We study some properties of the loss functions in the following.
\begin{Proposition}\label{proploss} Let us consider the observed data $(y_1(\omega), \ldots, y_n(\omega))$ fixed. Then the loss functions $\mathcal{L}_{L,\lambda, q}$ has the following properties

\begin{itemize}
    \item[] 1)For any choice of the hyperparameters $\lambda_j>0$, $j \in \{1, \ldots, m\}$ the value of the loss function is given by the least-square term if and only if f satisfies all the constraints. 

    \item[] 2)For fixed $L$ and $A_N$, for any choice of the hyperparameters $\lambda_j$, $j \in \{1,\ldots, m\}$, if $M_p \subset M'_p $ then the corresponding loss functions satisfy $\mathcal{L}^{'}\leq \mathcal{L}$.

    \item[] 3)For fixed $L$ and $M_p$, for any choice of the hyperparameters $\lambda_j$, $j \in \{1,\ldots, m\}$, if $A_N \subset \tilde{A}_N$, then $\tilde{\mathcal{L}} \leq \mathcal{L}. $
    
        \item[] 4)For fixed $A_N$ and $M_p$, for any choice of the hyperparameters $\lambda_j$, $j \in \{1,\ldots, m\}$, if $L_1 < L_2$, then $\tilde{\mathcal{L}_{L_1}} \leq \mathcal{L}_{L_2}. $
    
\end{itemize}
\end{Proposition}

\begin{proof}
1) For any choice of the hyper parameters $\lambda_j>0$, $j \in \{1, \ldots, m\}$, and fixed data observations $y_1,\ldots, y_n$ we have that if the polynomials satisfy exactly all the constraints we obtain zero contribution to the constraint terms. This gives the argument in one direction. For the reverse direction, if the constraints contributions is null then since they are summing squares that is non-negative functions the only possibility is for each one of them to be zero. Thus, the polynomial satisfies all the constraints.

2) As we are interested in a minimization problem, taking a minimum over a larger selection of ideals ($M_p \subset M'_p $) gives the fact that the minimum of the corresponding loss function $L^{'}$ can be at most as large as the minimum for $M_p$ but it can be in general smaller than that if the new set of ideals contains a better fit for the truth polynomial that generated the data.
Thus, as the set $A_N$ on which we select the polynomials from is fixed, we have $\mathcal{L}'\leq \mathcal{L}$.

3)As we fix the set of ideals we consider the minimum over, a larger set of selection of polynomials $\tilde{A}_N$ can decrease the minimum if the data is fixed. Thus,  $\tilde{\mathcal{L}} \leq \mathcal{L}. $

4) Similarly, if the search space $A_N$ and the ideals set are are fixed then considering $L_1 \leq L_2$ introduces new $L_2-L_1$ terms in the loss function which bring new contributions on top of the already existing terms that are shared by both the loss functions $\mathcal{L}_{L_1}$
and $\mathcal{L}_{L_2}$.

\end{proof}

\begin{figure}[H]
        \centering
        \includegraphics[width=.40\linewidth, trim=1.5cm 1cm 1.5cm 1cm, clip]{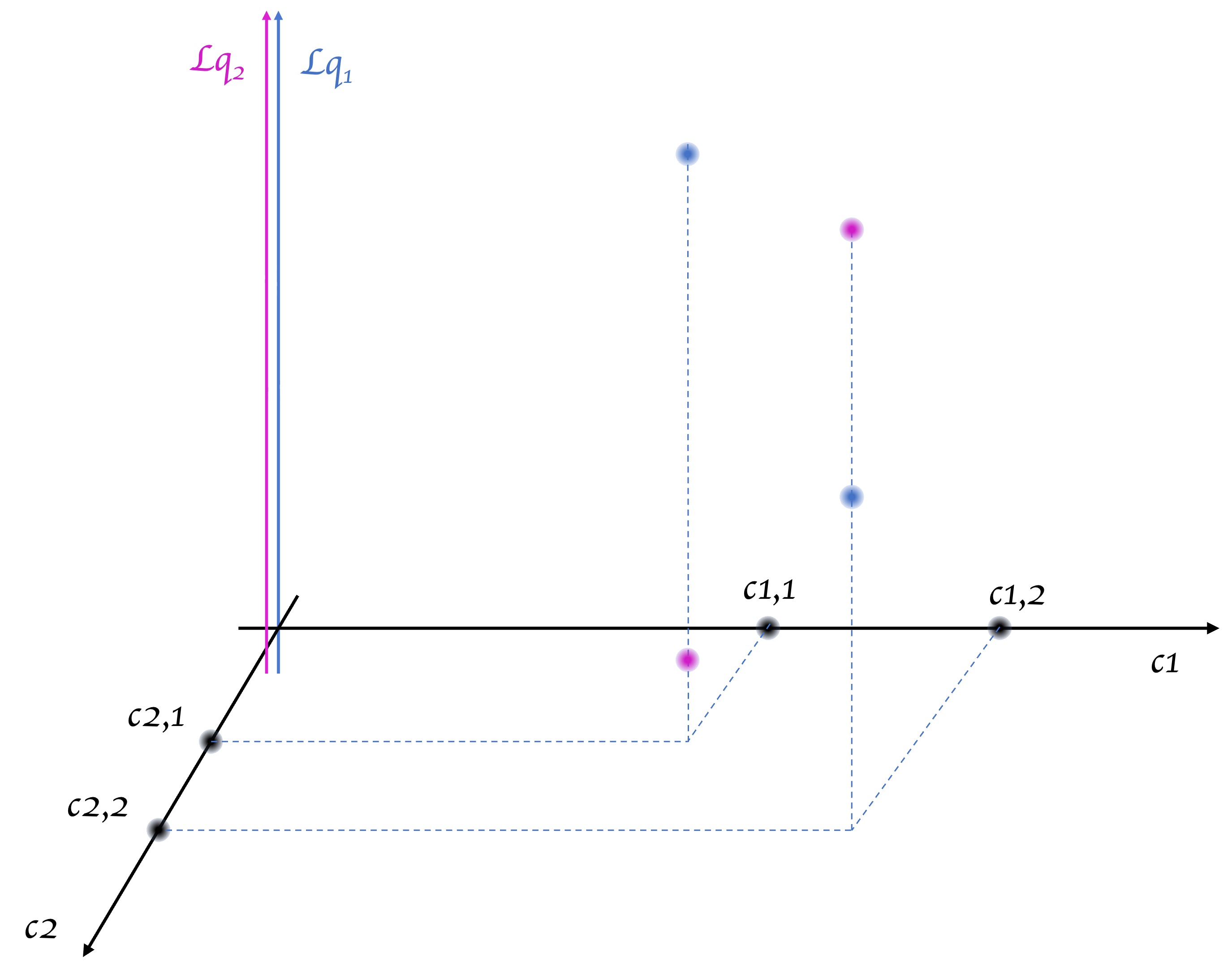}
        \caption{Schematic representation of the loss function landscape for two choices of  ideals $q$ and $q'$. The other indices in the definitions of the loss functions are discarded for convenience of the reader. 'Slices' of the loss functions are shown for two polynomials with coefficients $c_{1, i}$ and $c_{2, i}$, for $i \in \{1, 2\}.$ The z-axis for the two choices of the ideals are drawn on top of each other. }
\end{figure}

As mentioned in the one-variable case, one can considered versions of these loss functions for various prime numbers. In the following,for convenience of the reader, we return to the one-variable case for theoretical estimates and experiments. The analysis can be extended for multi-variables case.

\section{Probabilistic  techniques to study the minimizers in the one-variable case}

Our approach of recovery from noisy data allows one to obtain probabilistic estimates. We detail this in this section. In the one-variable case, we obtain probabilistic bounds for a certain polynomial to minimize the loss function. This is useful to know also how likely is for another polynomial to minimize the loss function which is not the true polynomial that was used to generate the noisy data. In addition, using union bounds one can obtain also information on the number of experiments needed in order for the probability that a minimizer is not the true polynomial to be certainly smaller than a certain threshold. For convenience of the reader we discuss the ideas in one variable  but one can extend these to multiple variables case.  The probabilistic techniques are explored in the context of derivations transformations but they hold for more general transformations and their corresponding axioms.

\subsection{Results for bounded noise} 
In this subsection, we explore some initial results on some simpler cases in which one considers data generated by a ground-truth polynomial which respects all the conditions in the penalty term in the loss function perturbed by bounded noise.
We study this bounded noise case via Hoeffding inequality and we obtain as a result  in this section the estimate in \eqref{boundHeoff}. This analysis can be extended further to more involved penalty terms and to multi-variables. Since this is the first work on this direction, we perform the analysis in this case to exemplify the method. We will drop the indices from the definition of the loss functions.

For a fixed competing polynomial, define $ d_j = f(x_j)-f_*(x_j),$ and $ D_f^2 = \frac{1}{n} \sum_{j=1}^n d_j^2. $
Let $\Delta_f = R(f)-R(f_*).$  Since $y_j=f_{*}(x_j)+\varepsilon_j,$ we have 

\begin{align} \mathcal{L}(f)-\mathcal{L}(f_*) = \frac1n \sum_{j=1}^n \left( \bigl(\varepsilon_j-d_j\bigr)^2 - \varepsilon_j^2 \right) + \Delta_f = D_f^2 + \Delta_f - \frac{2}{n} \sum_{j=1}^n \varepsilon_j d_j. 
\end{align}

Therefore, we have that $\mathcal{L}(f)\le \mathcal{L}(f_*) $ is equivalent to 
$\frac{2}{n} \sum_{j=1}^n \varepsilon_j d_j \ge D_f^2+\Delta_f.$ In our model, we assume independence, $\mathbb E[\varepsilon_j]=0$, and $|\varepsilon_j|\le b.$ Let us also define $ X_j = \varepsilon_j d_j.$
Then $\mathbb E[X_j]=0,$ and \(X_j\) lies in an interval of length $ 2b|d_j|.$
Given the definition of our loss functions, we have $ D_f^2+\Delta_f>0.$ Then, we obtain $$\mathbb P \left( \mathcal{L}(f)\le \mathcal{L}(f_*) \right) = \mathbb P \left( \sum_{j=1}^n X_j \ge \frac{n}{2} \bigl(D_f^2+\Delta_f\bigr) \right). $$ 
Hoeffding's inequality gives $$ \mathbb P \left( \sum_{j=1}^n X_j \ge \frac{n}{2} \bigl(D_f^2+\Delta_f\bigr) \right) \le \exp \left( - \frac{ 2 \left( \frac{n}{2} \bigl(D_f^2+\Delta_f\bigr) \right)^2 }{ \sum_{j=1}^n \left(2b|d_j|\right)^2 } \right).$$ 
Using that $\sum_{j=1}^n d_j^2 = nD_f^2,$ we obtain $ \sum_{j=1}^n \left(2b|d_j|\right)^2 = 4b^2nD_f^2.$
We thus have that 
\begin{equation}\label{boundHeoff}
\mathbb P \left( \mathcal{L}(f)\le \mathcal{L}(f_*) \right) \le \exp \left(- \frac{ n\bigl(D_f^2+\Delta_f\bigr)^2 }{ 8b^2D_f^2 } \right), \hspace{2mm} D_f^2+\Delta_f>0. 
\end{equation}

We note that, under the assumption $R(f_*)=0$ and $ R(f)\ge0,$ we have $ \Delta_f = R(f)\ge0.$

 \subsection{Analysis of the hyperparameter} 
 
In this section, we study a particular example in which the penalty term in the loss function is composed only of three terms. The analysis done in this particular case can be further extended to more involved cases.  Let us assume now that the true polynomial satisfies all three arithmetic constraints appearing in the penalty, that is, for the true polynomial we have that $q_1(f_*)=0$, $q_2(f_*)=0,$  $q_3(f_*)=0. $ This gives that $R_{\Lambda}(f_*)=0$. For a competing polynomial $f$ in the set $A_1$ under consideration, let us consider $q(f) = \bigl(q_1(f),q_2(f),q_3(f)\bigr),$ so that $ R_{\Lambda}(f)=\Lambda\cdot q(f),$ where, we consider for simplicity the model in which we have a unique hyperparameter $\lambda$ so $\Lambda=(\lambda, \lambda, \lambda). $
 As before $d_j = f(x_j)-f_*(x_j)$, $D_f^2 = \frac1n \sum_{j=1}^n d_j^2,$ and $Z_f = \frac{2}{n} \sum_{j=1}^n \varepsilon_jd_j. $
Then $$\mathcal{L}_{\Lambda}(f) - \mathcal{L}_{\Lambda}(f_*) = D_f^2+\Lambda\cdot q(f) - Z_f.$$ 
We thus have $\mathcal{L}_{\Lambda}(f_*) \le \mathcal{L}_{\Lambda}(f)$ if and only if $Z_f \le D_f^2 +\Lambda\cdot q(f). $
\noindent\\
 Using that $|\varepsilon_j|\le b,$  we have $ Z_f = \frac{2}{n} \sum_{j=1}^n \varepsilon_jd_j $ satisfies $Z_f \le \frac{2b}{n} \sum_{j=1}^n |d_j|. $ 
 \noindent
 
Therefore a competing polynomial $f \in A_1$ 
 can't have a lower loss function when evaluated on it than \(f_*\) whenever $ D_f^2+R_{\Lambda}(f) > \frac{2b}{n} \sum_{j=1}^n |f(x_j)-f_*(x_j)|. $ 
 
In our case, a condition for \(f\) to be impossible as a minimizer is $$ \lambda Q(f) > \frac{2b}{n} \sum_{j=1}^n |f(x_j)-f_*(x_j)| - D_f^2.$$ 
 
One can then consider a critical value associated with \(f\), whenever $Q(f)>0$, by $$\lambda_f^{\mathrm{critical}} = \frac{ \left( \displaystyle \frac{2b}{n} \sum_{j=1}^n |f(x_j)-f_*(x_j)| - D_f^2 \right)_+ }{ Q(f) }.$$
 
Increasing $\lambda$ eventually removes individual incorrect polynomials from the set of possible minimizers.

In addition, one can consider unbounded cases such as  Gaussian noise version of this problem and obtain a corresponding estimate in that case.

\begin{rmk}
In addition, these estimates interact with the structure of the axioms of the transformations involved. 
Indeed, 
$$\mathbb P \left( \mathcal{L}(f+r)\le \mathcal{L}(f_*) \right) \le \exp \left(- \frac{ n\bigl(D_{f+r}^2+\Delta_{f+r}\bigr)^2 }{ 8b^2D_{f+r}^2 } \right), \hspace{2mm} D_{f+r}^2+\Delta_{f+r}>0.$$ 
As the term $\Delta_{f+r}$ contains mixed applications of p-derivations and usual derivatives then the structure of the axioms of these transformations interact with the exponent in the estimate. As in the loss function discussion from before, the signs of the coefficients in the polynomial $r(\cdot)$ matter as this will lead to positive and negative signs that might produce cancellations so a careful analysis of these estimates is required. The additive example $f+r$ is just a particular choice but one can consider other versions such as $f+r^2,$ etc.

In addition, It is interesting to consider families of polynomials in a certain region and then ask what is the chance that there exists a polynomial inside this family for which the loss $\mathcal{L}(f)$ is less than the loss on the true polynomial. One can consider then a union bound estimate. One is interested in how big are the corresponding probabilities to know how many terms one can take in order for the probability to not surpass a given threshold. 

\end{rmk}

\subsubsection{General transformations in this context}

One can consider also a more general setting in which the transformations T are given by some rules (they might not be p-derivations or standard derivatives). Then, the estimates will depend on the structure of these transformations and their corresponding axioms. 
Indeed, one can consider models in which in the penalty terms appear various transformations (the list is non-exhaustive):
\[
\begin{aligned}
T(f+g)&=T(f)+T(g),\\
T(fg)&=T(f)g+fT(g),
&&\text{ordinary derivation},\\
T(f+g)&=T(f)+T(g)+C_p(f,g),
&&\text{$p$-derivation},\\
\end{aligned}
\]

The constraints of the truth polynomial might be represented also by other types of transformations, for example, a certain algebraic relation,
$$
T(f)=F(\sigma,f,\partial f),
$$
for some function or parameter $\sigma.$ 

One can phrase these models of recovery from noise also as partial information leaked to an attacker model. Indeed, one is interested to see if an attacker which has access to the noisy data and has access to some partial information  can recover the true signal. In the last example, the leaked information can be modeled as
$$
\|F(\sigma,f^*,\partial f^*)\|\le\gamma,
$$
for some real-valued $\gamma>0$.
In the next section, we perform experiments in this direction.
Also, one can  consider a model in which an attacker tries several transformations for the given data and ranks them according to the score in terms of the loss function of the minimizer obtained for each transformation model.

\subsection{Applications of this $ALS_d$: Data-driven Statistical Arithmetic Geometry, Enumerations and Incidence}

Taking the view in the preface of the book \cite{Arithmeticgeometrybook}, we see Arithmetic Geometry as solutions of systems of polynomials with coefficients in $\mathbb{Z}/p\mathbb{Z}=\mathbb{Z}_p$, $\mathbb{Z}$, or $\mathbb{Q}$. Then, if one is interested in a problem of the type given a system of $N$ polynomial equations with  $N-k$, $k <N$, of them and the other subjected to uncertainty due to errors in the transmission of the data, one encounters problems of Random Arithmetic Geometry. We refer the reader to \cite{Intersectiontheoryfultonbook}, \cite{Arithmeticgeometrybook} for more details on the deterministic formulation of these directions. See also \cite{Modernenumerativebook} for a deterministic Enumerative Geometry reference. We focus on the case of integer-coefficients polynomials and their reduction mod p, for some prime number p.

There are different approaches to consider random models: 

\subsection{Random polynomials models}

In one direction, one can consider models of random polynomials and study the corresponding probabilistic versions of the deterministic theories.
In Random Intersection Theory one would be interested in counting the number of solutions of such systems of random polynomial equations and estimates  on those. Random Incidence Theory (see \cite{Introincidence} for a deterministic approach) would be closer in spirit to this. We take the general view that one can think of Random Arithmetic Geometry and the corresponding Enumerative Geometry as solutions and counting solutions to systems of random polynomial equations with coefficients in $\mathbb{Z}_p$, for some prime number p, or $\mathbb{Q},$ or $\mathbb{Z}.$

$$
\begin{cases}f_1(x_1, x_2, \ldots, x_N;\omega)=0\\f_2(x_1, x_2, \ldots, x_N;\omega)=0\\ 
f_3(x_1, x_2, \ldots, x_N;\omega)=0
\end{cases}
$$

In the above example, for simplicity, we took only three polynomial equations described be the random polynomials $f_i(\cdot)$ (we can consider more general versions of these) with random coefficients in $\mathbb{Z}_p, $ for some prime number p, $\mathbb{Q},$ or $\mathbb{Z}.$
One natural model is composed of polynomials with coefficients from a uniform distribution but one can consider various models for the randomness of the coefficients.

\subsection{Data-driven approach}

In a different direction, in the context of the ALS one is interested in the data-driven versions of these in which one observes some noisy real data, and wants to recover structure from this setting. According to our search, this approach was not considered so far in the literature. In the example that is emerging from the ALS  we study, it is interesting to note that the corresponding random polynomial systems are specific but one can take the general view mentioned above for the fields of Random Arithmetic Geometry and  Enumerative Geometry. 
To simplify the general picture, let us look at the following system of polynomials in which the first two are deterministic and for the third one has access to noisy data only (in  terms of pairs $(x_i, y_i(\omega))$ about it.

$$
\begin{cases} 
f_1(x_1, x_2, \ldots, x_N)=0\\f_2(x_1, x_2, \ldots, x_N)=0\\ 
f^\sim_3(x_1, x_2, \ldots, x_N)
\end{cases}
$$
One can consider other more general models with more equations in which part are deterministic or random from given probabilistic model and the remaining ones are recovered from noisy data, etc.

For this system or other alike,  one can consider the reduction of the system mod p, for some prime p (and repeat for more primes) and if one knows that one or more of the polynomials of the systems belongs to some symbolic powers as in our analysis so far, then one knows information about their derivatives behavior etc that is useful when studying these systems. For example, one can take further derivatives and reduce mod p, etc. and study the system. 

The first part (Subsection 4.4) we leave unexplored in this work. We aim to explore it in future work.  We focus on the second part.

\subsubsection{Experiments in Claude on Recovery from Noisy data at unseen noise levels}

This development can be thought as a version of Arithmetic Machine Learning. We study experiments in which we train parameters on some portion of the data and we test them on the remaining portion of unseen values.

In our analysis, we consider the Jaccard distance on varieties. One can consider other more refined notions in future work.
\begin{mydef}
  For two finite varieties $V_1, V_2$, the \emph{Jaccard distance} is
  \[
    d_J(V_1, V_2) \;=\; 1 - \frac{|V_1 \cap V_2|}{|V_1 \cup V_2|},
  \]
  with the convention $d_J(\emptyset, \emptyset) = 0$.
\end{mydef}

For each noise realization $\omega$, the fitting procedure produces the polynomials
in $\mathbb{Z}[x]:$ $f^*(x),\; \hat{f}^\omega_{\mathrm{reg}}(x),\;
  \hat{f}^\omega_{\mathrm{unreg}}(x).$
In our context, we will be interested in the following varieties
$$
  V^* \;=\; \bigl\{x:\, f^*(x) \equiv 0 \hspace{1mm}\text{mod}(p)\bigr\}.
$$
$$
  V_{reg}^{\omega} \;=\; \bigl\{x:\, \hat{f}^\omega_{\mathrm{reg}}(x) \equiv 0 \hspace{1mm}\text{mod}(p)\bigr\}.
$$
$$
  V_{unreg}^{\omega}\;=\; \bigl\{x: \hat{f}^\omega_{\mathrm{unreg}}(x) \equiv 0  \hspace{1mm}\text{mod}(p)\bigr\}.
$$


In the next examples, the hyperparameter is learned as well from the data. The experiments in this section are performed with Alternating Search Method. The anchor $a$ in $(p, x-a)$ is manually selected. We note that the conditions in the penalty terms are local (in the sense that they are formulated at one point) but the recovery problem is formulated for the global variety. In this initial experiments, we aim to study if the rigidity of the objects can lead from local information to global recovery. It is interesting to study this phenomena further for more involved scenarios and many variables with more computational power. We take the view that noisy data points with prior/leaked information to the investigator (such as the underlying polynomial might be close to some symbolic power of certain ideal) are studied along with some other family of integer-coefficients polynomials. Then, considering mod p of those and partial recovery of the variety globally from the data, leads to studies of Data-driven Intersection and Enumeration Geometry mod p for the whole system (the previous polynomials and the recovered variety from the noisy data of the additional one).

\begin{figure}[H]
        \centering
        \includegraphics[width=.38\linewidth]{ 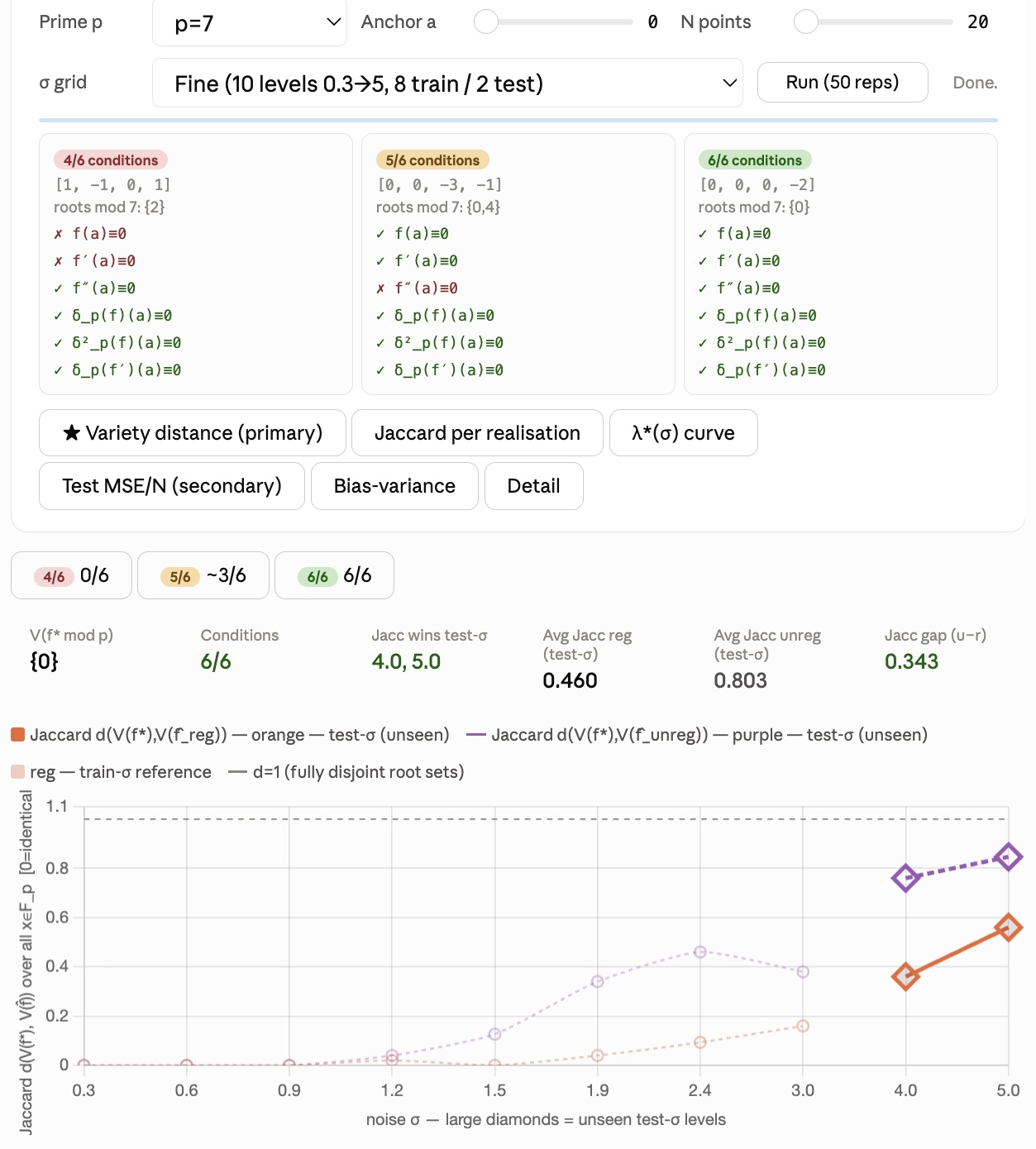}
        \caption{Jaccard distance comparison between recover data from noise and the true object at unseen noise levels. The data is generated from a polynomial satisfying all 6 conditions in the arithmetic penalty to which we add noise. The hyperparameter $\lambda$ is learned from parts of the data. We use a 80-20 \% data splitting. In the image the first values of the noise parameter, we use for training and the remaining 20\% is for testing.}
\end{figure}

\begin{figure}[H]
        \centering
        \includegraphics[width=.38\linewidth]{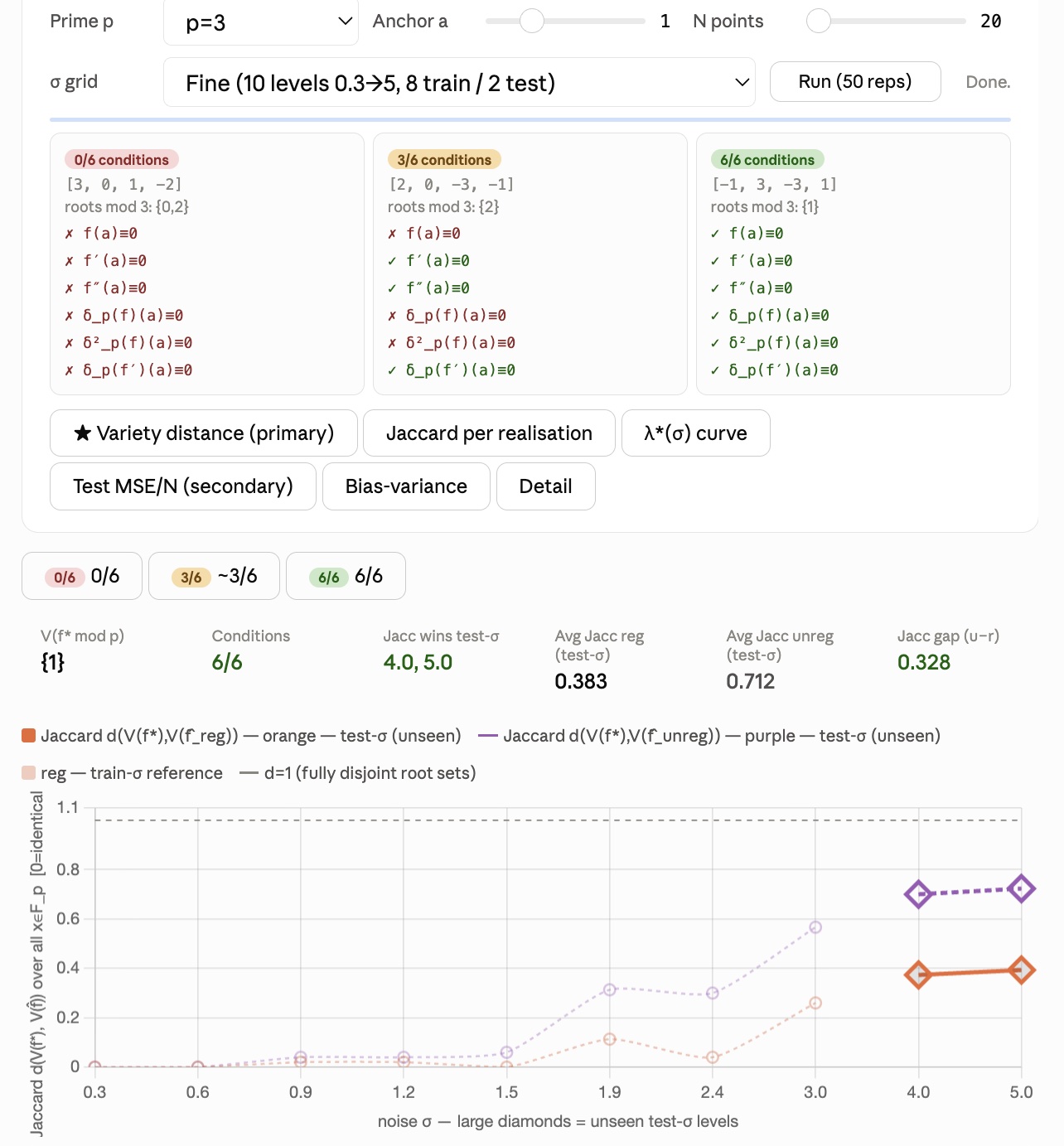}
        \caption{Jaccard distance comparison between recover data from noise and the true object at unseen noise levels. The data is generated from a polynomial satisfying all 6 conditions in the arithmetic penalty to which we add noise. The hyperparameter $\lambda$ is not manually set but is learned from parts of the data. We use a 80-20 \% data splitting. In the image the first values of the noise parameter, we use for training and the remaining 20\% is for testing.}
\end{figure}

\subsubsection{Comparison between two methods}

We also study the performance of Alternates Search (AS) and Simulated Annealing (SA) of recovery on unseen noise values.  
We consider models in which the underlying true polynomial respects none, few, or all the conditions in the regularizer term. In the experiments, we learn the parameter $\lambda$ on parts of the data, and we test with that parameter on unseen noise level two methods: Simulated Annealing and Alternating Search. We observe that $\lambda (6/6) > \lambda (3/6) > \lambda(0/6)$, thus the learned parameter penalizes more heavily when the data is generated by a true polynomial which satisfies more of the 6 conditions in the penalty term.

\begin{figure}[H]
        \centering
        \includegraphics[width=.65\linewidth]{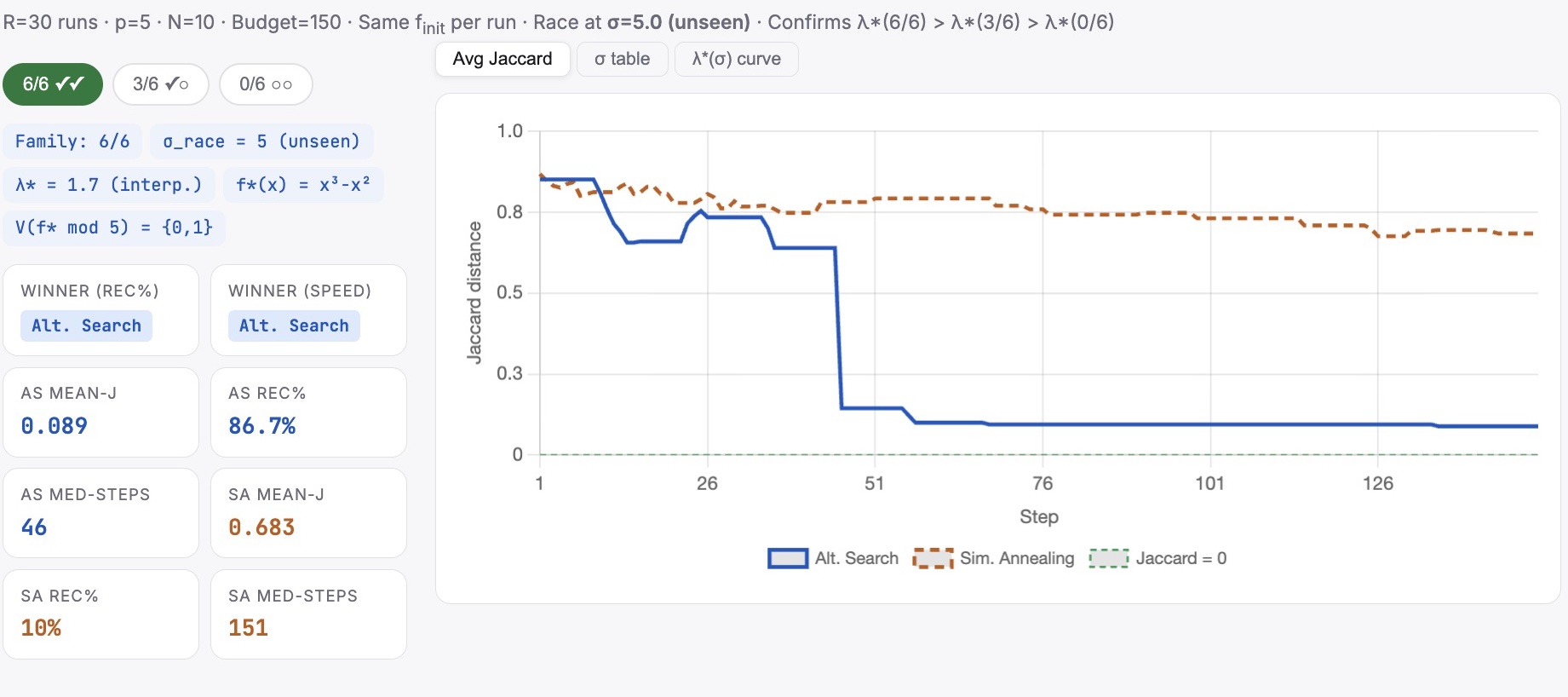}
        \caption{Jaccard distance comparison between two approaches Alternate Search and Simulated Annealing for unseen noise with learned parameter $\lambda$, when data is generated by a true polynomial that respects all 6 conditions. }
\end{figure}

\begin{figure}[H]
        \centering
        \includegraphics[width=.65\linewidth]{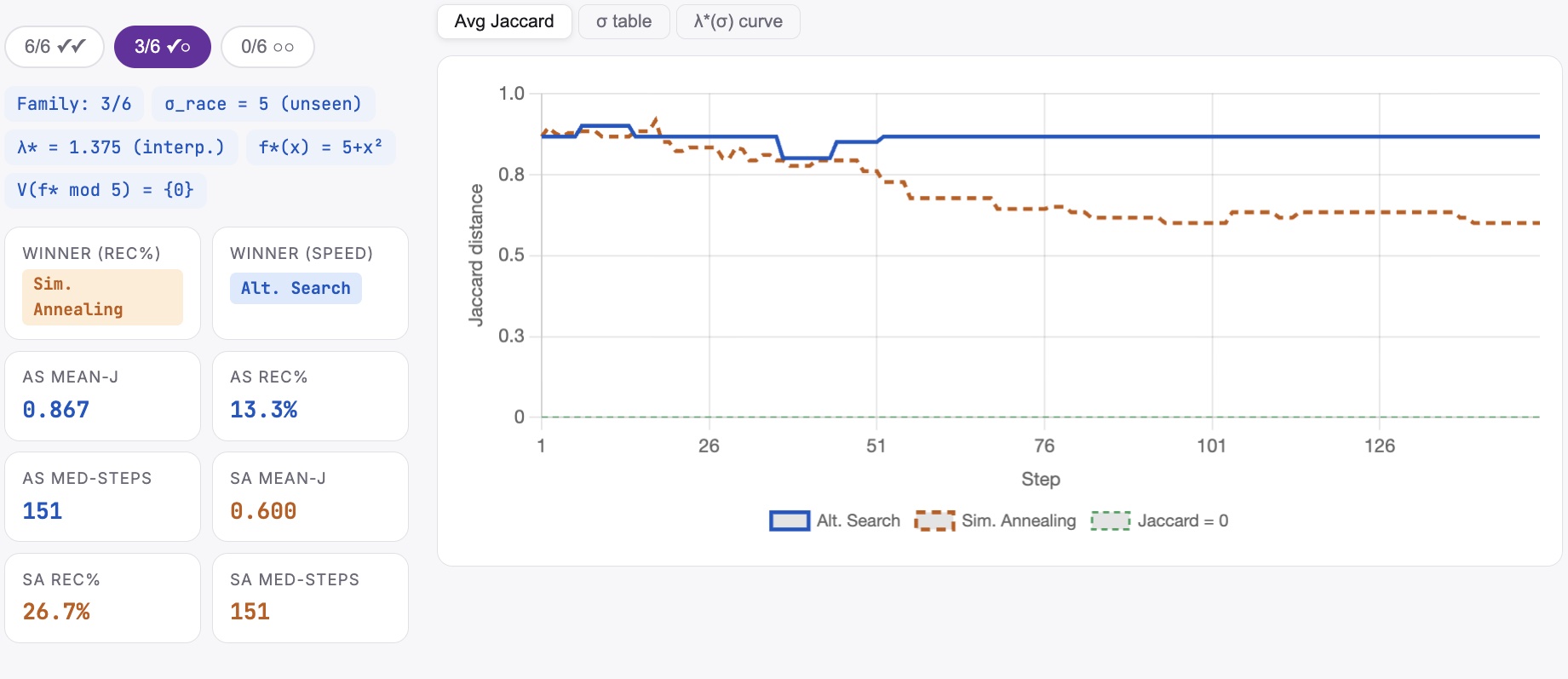}
        \caption{Jaccard distance comparison between two approaches Alternate Search and Simulated Annealing for unseen noise with learned parameter $\lambda$, when data is generated by a true polynomial that respects 3/6 conditions. }
\end{figure}

\begin{figure}[H]
        \centering
        \includegraphics[width=.65\linewidth]{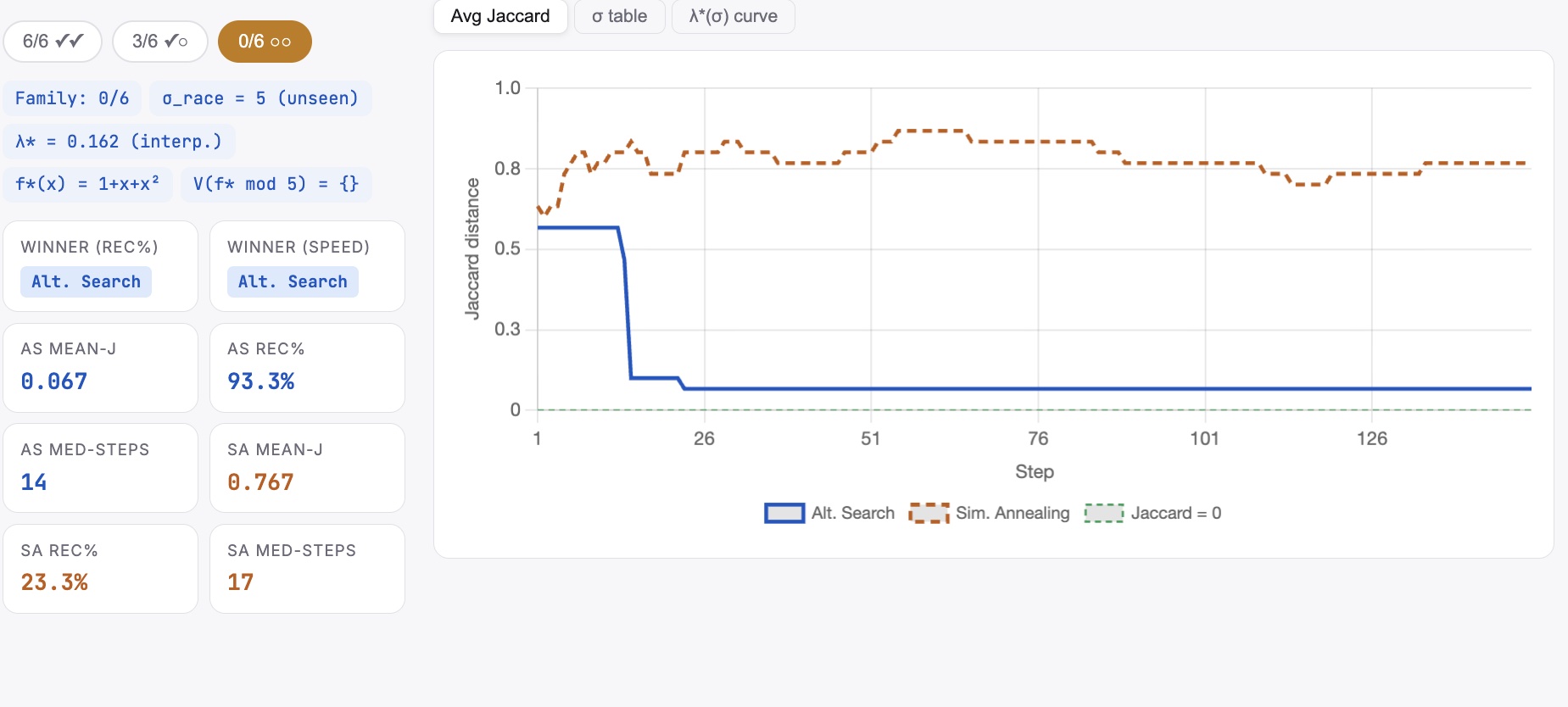}
        \caption{Jaccard distance comparison between two approaches Alternate Search and Simulated Annealing for unseen noise with learned parameter $\lambda$, when data is generated by a true polynomial that respects  0/6 conditions. }
\end{figure}

\subsubsection{Further potential interpretations in terms of Data Security}

With the development of Machine Learning models one is interested in problems related also with Data Security. We note that given one of the goals to do is recovery from noisy data, one can consider a problem similar in spirit with the setting of Learning with Errors in which one considers noisy polynomials with cryptographic applications. As interpolation and data-fitting of noisy data are natural questions for Machine Learning this model might lead to new questions related potentially with security of data. See \cite{regev2009lattices}, \cite{partialleak}. The setting we consider is similar in the sense that we have access to noisy data generated by a true underlying polynomial that is perturbed by noise. We study only the variety recovering. The experiments in the previous sections can be reinterpreted in the following language. We took the view that side information might have been leaked to an attacker. Then, this attacker might have partial information that the true polynomial whose perturbation generated the data might be close to being in symbolic power (having most of the conditions satisfied, .etc) or any other versions of these. Then, we put this information in the penalty terms of the loss functions, and study if such a model potentially used by an attacker could recover features of the true underlying polynomial in various contexts. In our previous experiments, we studied the variety recovery, not the coefficients recovery, which is a weaker recovery. However, if a potential attacker recovers from a local condition (constraints in one point) information about the  whole variety then it can limit the search space reducing the complexity of the problem of search.

One can study similar problems for more complicated scenarios in multiple variables including coefficient recovery in future works with more computational power available.

\subsubsection{Interpolation with and without noise}

In the area of Machine Learning, one is interested also in interpolation problems (see \cite{murphy2012machine}). In our context, then one is interested to see if integer-coefficients polynomials (along with their higher-order derivatives, higher order p-derivations, and their mixes) when evaluated at certain integer points $a \in \mathbb{Z}$ are a multiple of the prime number $p.$  These can be interpreted as interpolation conditions at points $x=a$, that is $f(a)=kp$ for some $k \geq 0$. Similarly for the higher order derivatives and high order derivations and their mixes. We note that there appear in the literature studies in which loss functions which contain both least-squares and an interpolation terms are considered. For some models, the interpolation terms require the value of the gradient (related with standard derivatives) in certain points or regions. See \cite{CSimagereconstructpaperrobust}, the work which starts at page 532 in \cite{booksurfacereconstructionpage5322015scale}, and the references therein. See also \cite{veryimportantsoftandhardlagrangepaper2023augmented}. In addition, \cite{murphy2012machine} considers both interpolation and noisy interpolation problems.

One might face the challenge that a polynomial of a certain degree can't pass through all of the observed noisy points (depicted in the schematic representation with red). Thus, one is interested to find the polynomial that fits the data in (in blue) and passes close to the red points or through one (as in the picture) or more of them if that is possible.

\begin{figure}[H]
        \centering
        \includegraphics[width=.40\linewidth]{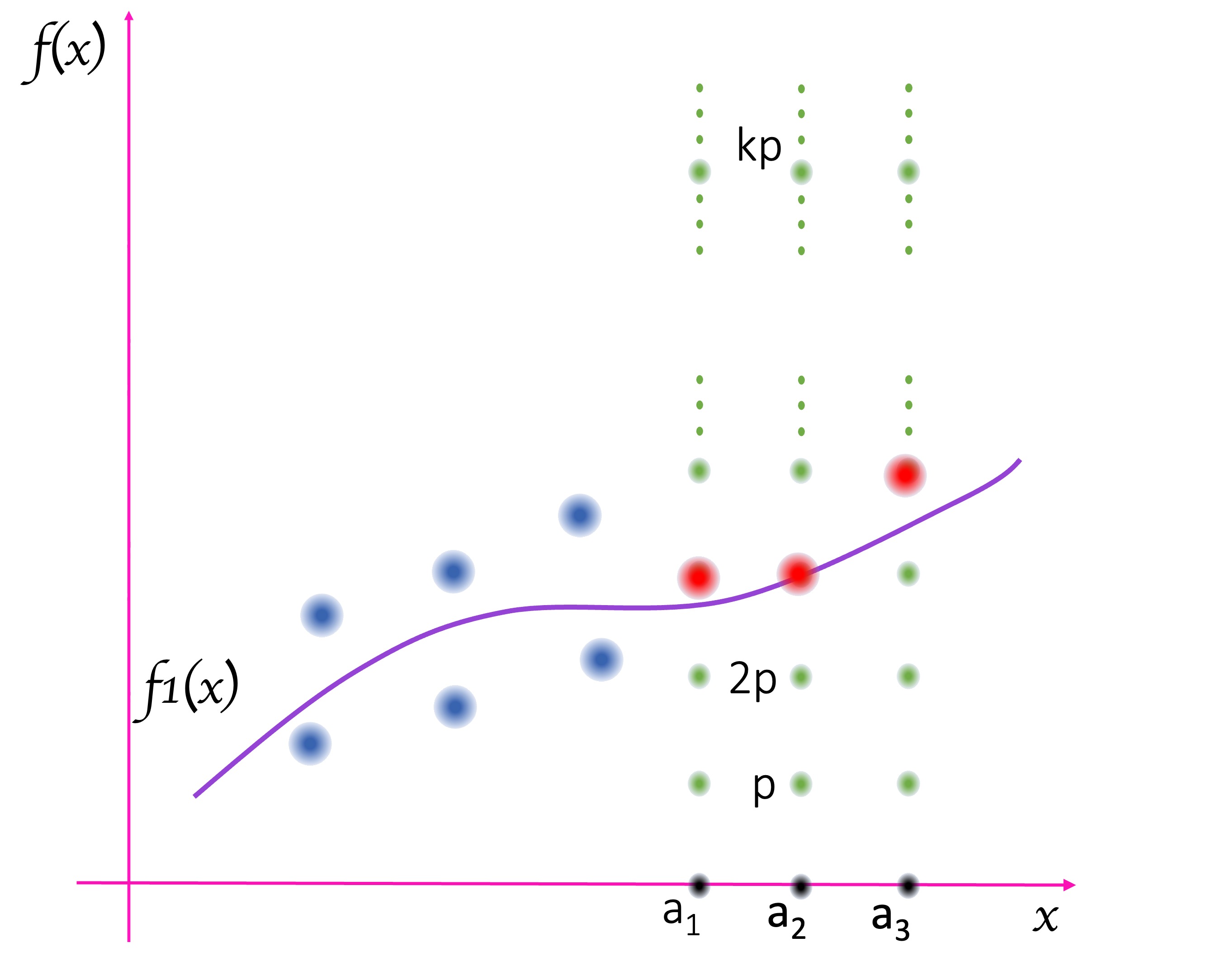}
        \caption{Schematic representation of noisy interpolation.}
\end{figure}

\section{An $ALS_d$ based on Hasse-Schmidt Derivations}

In addition to the previously studied in detail example of a $ALS_d$, we discuss in this section shortly another interesting example. The previous example used certain choice of losses in the total loss function. In this one, we consider a notion of distance for convenience. For this, let us first introduce the following distance on $\mathbb{Z}_p,$ for $p$ a prime number. 
\begin{mydef}
For $\hat{x}$ and $\hat{y}$ $\in$ $\mathbb{Z}_p$, let us consider

\begin{equation}
d(\hat{x},\hat{y})=\min(|\hat{x} - \hat{y}|, p - |\hat{x}- \hat{y}|).
\end{equation}
\end{mydef}



Let us consider an $ALS_d$ based on the Hasse-Schmidt derivations.  See \cite{hasse1937noch}. Let $k[x]$ be a polynomial ring over a field $k$. The $l$-th Hasse derivative of $x^n$ is

$$
D^{(l)} x^n=\binom{n}{l} x^{n-l},
$$

if $n \geq l$ and zero otherwise. Then, this is extended $k$-linearly to polynomials on $k[X].$ 

In characteristic zero, the Hasse-Schmidt derivation takes the form

$$
D^{(l)}=\frac{1}{l!}\frac{\mathrm{d}}{\mathrm{~d} x^l}.
$$

The Hasse-Schmidt derivation is a generalized derivation on $k[x]$ which satisfies an analogue of the product rule

\begin{itemize}
\item[] a) Axiom/Rule: For f, g $\in k[X]$, we have
$$
D^{(l)}(f g)=\sum_{i=0}^l D^{(i)}(f) D^{(l-i)}(g),
$$
\end{itemize}

and an analogue of the chain rule which can be thought as another axiom/rule when working with these transformations in the context of this $ALS_d.$ 
These objects are interesting also since there is a  Taylor's theorem involving them. In addition, their specific axioms/rules reflect the

In our example, we choose $k=\mathbb{Z}_p$, for some prime p. Let us consider $ \hat{f}(x) \in \mathbb{Z}_p[x]$.
We then consider the $ALS_d$ with the following sequence of loss functions

\begin{align*}
& \quad \mathcal{L}_{l, \lambda, \hat{u}}(\hat{f}, \omega)=\frac{1}{n}\sum_{i=1}^n (d(\hat{f}\left(\hat{x}_i\right),{\hat{y}_i}(\omega)))^2+\lambda (d(D^{(l)} \hat{f}(\hat{u}), \hat{0}))^2 \\
& \mathcal{L}_{l, \lambda}^*=\min_{\hat{u}\in \mathbb{Z}_p\cap A} \min_{ \hat{f} \in \mathbb{Z}_p[x]\cap A^*_1} \mathcal{L}_{l, \lambda, \hat{u}}(\hat{f}, \omega).
\end{align*}

In the above, $A$ can be chosen to be $A=\mathbb{Z}_p$ or a subset of it, and $A^*_1$ is defined similarly with $A_1$ for polynomials in $\mathbb{Z}_p[x].$

In addition, one can study the case corresponding to various multiplicities at different points. Let $j \in \{1, 2, \ldots, J\}.$

\begin{align*}
& \quad \mathcal{L}_{l, \lambda, \hat{u}_J}(\hat{f}, \omega)=\frac{1}{n}\sum_{i=1}^n (d(\hat{f}\left(\hat{x}_i \right),{\hat{y}_i}(\omega)))^2+\lambda \sum_{j=1}^J (d_T(D^{(l)} \hat{f}(\hat{u}_j), \hat{0}))^2 \\
& \mathcal{L}_{l, \lambda, \hat{u}_J}^*=\min_{\hat{u}_1\in \mathbb{Z}_p,\hat{u}_2\in \mathbb{Z}_p,\ldots, \hat{f} \in \mathbb{Z}_p[x] \cap A^*_1} \mathcal{L}_{l, \lambda, \hat{u}_J}(\hat{f}, \omega).
\end{align*}



\subsection{Other examples of $ALS_d$}

So far, we have studied datasets consisting of structured objects that were linked with Arithmetic. The natural objects along with their operations were the integer coefficients polynomials in several variables. 

One can consider other structured examples. In the example below the polynomials are remainders with respect to the division of the generator of the ideal generated by $y^2-x^3$. Given thus, these polynomials have a certain structure (not every format of the polynomial is allowed). Then, when one considered optimization problems over spaces of objects with this format the loss function is evaluated on these types of polynomials so to move between the objects one needs to keep the structure in mind. We also investigate the corresponding derivations on these spaces. We consider a data-set made of objects from the following quotient ring $R:=k[x,y]/(y^2-x^3).$
Derivations with respect to the field $k$, on this ring are given by
\[
\delta_1 \;=\; 2y\,\partial_x + 3x^2\,\partial_y,
\qquad
\delta_2 \;=\; x\,\partial_x+\frac{3}{2}y\,\partial_y.
\]

One can then in turn study $ALS_d$ with applications in \textit{Algebraic Machine Learning.}

\section{Transformations-Axiomatic Based Learning (TABLe) and Conjectures generation}

We now do an initial study of the reverse direction. For this, we consider in the following the idea of tables/libraries of transformations and their axioms. 
The developments in \cite{Davies2021Advancing} are in a different direction. The key distinction is that our approach is transformations-axiomatic-based. We systematically search over mathematically meaningful transformations and their rules, rank them by how well they characterize the sampled mathematical objects, and use the successful/failed transformations to suggest new structural characterizations or conjectures. Although we consider an initial experiment linked with ideals, this method can be replicated on Mathematical objects data-sets related with various other areas of research.

Our approach is not meant to replace human researchers but to leverage the capabilities of Machine Learning to generate conjectures and useful characterization of spaces of objects. In addition, we use the scores to rank the transformations and their axioms in terms of close they are to characterize the objects under consideration. Then, one can do a more detailed approach on the region which best characterizes the data by modifying the format of the axioms and check in that region if we get better predictions by increasing the scores. In some sense, this is close  to the thinking of humans when trying various transformations/methods over problems of research/data under consideration. The proposed approach does not lead directly to the potential solution but also contains a detailed summary of the scores various approach received, such that one can modify the axioms accordingly in the regions with high scores and to a more detailed search there.

We perform a controlled experiment  in the context of transformations that characterize the symbolic powers for ideals containing the prime number p. For this we generate a subset of the elements consisting of polynomials up to degree 3 and with bounded integer coefficients between [-6, 6]. This leads to the results below for a class of transformations considered. We emphasize that in contrast with another approach developed in the literature \cite{brunton2016discovering} we do not consider dynamical picture as our approach is transformation-transformation-axiomatic based and is centered around axioms and rules of transformations that apply directly to the objects. In addition, one can  study other datasets and their corresponding structures using  analogies with the methods developed in this paper. The same ideas that are developed here can be studied for other structures by analogy. 

Following the previous discussion, we know introduce these ideas more formally. 
Suppose we generate by brute force from the algebraic definition, a dataset

$$
D=\{f_1,\ldots,f_N\}
$$

of exact mathematical objects. For example, $f_i\in Q^{(n)}$, for a symbolic power $Q^{(n)}$. 

Then, consider a table/library of candidate transformations 

$$
\mathcal T=\{T_1,\ldots,T_M\},
$$

which might contain ordinary derivatives, $p$-derivations, compositions, difference operators, evaluation maps, etc. The transformations could satisfy different rules (the list is non-exhaustive):
\[
\begin{aligned}
T(f+g)&=T(f)+T(g),\\
T(fg)&=T(f)g+fT(g),
&&\text{ordinary derivation},\\
T(f+g)&=T(f)+T(g)+C_p(f,g),
&&\text{$p$-derivation},\\
\end{aligned}
\]
where $C_p(\cdot, \cdot)$ is the correction factor appearing in the sum rule for the $p$-derivations. Moreover, one can consider other versions in which the correction terms are indexed by some function or scalar parameter $\alpha C_p(\cdot, \cdot)$ which one would like to deduce from the scores by fine-tuning in the areas with high scores until one finds a suitable choice. Similar ideas can be studied for general transformations and their axioms.
The symbolic-power example is a good controlled example.  One can start with the algebraic definition of $Q^{(n)}$ and assume it does not know the differential characterization theorem. Using the algebraic characterization first, one can generate by brute force search some elements of the symbolic power of an ideal. Then one hopes for another characterization of these elements through some transformations. The Transformation-Axiomatic based approach then suggests searching over family of transformations (and assign scores to them in terms of how well they obtain the known characterization):

$$
\mathcal T=\{\partial_x,\delta_p,
\delta_p\partial_x,
\partial_x^2,
\delta_p^2,\ldots\}
$$

and to progressively discover which particular collection of transformation conditions separates membership from non-membership. The same can be replicated for multi-variables case including partial derivatives, etc. We consider the one-variable case for simplicity.  Following the characterization of the transformations for ideal membership obtained in \cite{Jeffriescite} (see also Theorem 5 in \cite{AMSdifbyprime}) allows us we consider the following controlled experiment. 

Similar ideas can be explored for various objects in search for characterization of their underlying structure  in other areas. 

\begin{figure}[H]
        \centering
        \includegraphics[width=.55\linewidth, ]{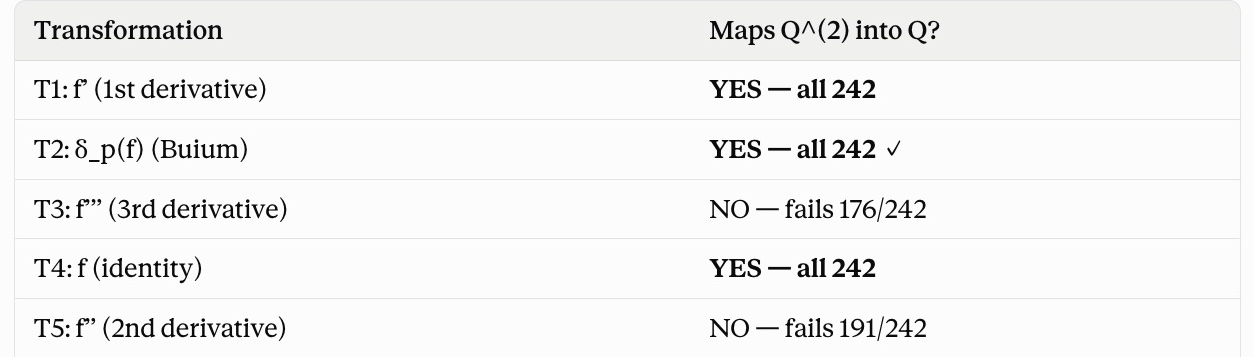}
        \caption{Experiments involving a subset of elements of the symbolic powers of order 2, for the prime $p=5$, for the ideal $(5, x-3)$. We considered a family of transformations on these elements and we checked their scores also on the right. This experiment serves as an example for general data-sets and for general transformations with their corresponding formulas and axioms to be studied. The ratio obtained can be interpreted as scores. Similarly with mathematical thinking, we learn through those scores how close were the other approaches given by the other transformations to characterize the set.} 
\end{figure}

\section{Gibbs measures and their interactions with the Axioms}

We start with the Gaussian Free Field as a very important model of Statistical Physics.

\subsection{The Gaussian Free Field model}
We consider a configuration $\sigma \in \Omega$, in which a variable
$\sigma_i \in \mathbb{R}$ is associated to each vertex $i \in \mathbb{Z}^d$; as usual,
we will refer to $\sigma_i$ as the spin at $i$. We define the interactions between
the spins located inside a region $\Lambda \subset \mathbb{Z}^d$, and between these
spins and those located outside $\Lambda$. 
The model has an energy that is based on a Hamiltonian which has an interaction term which favors agreement of
    neighboring spins. 
The local agreement between neighboring spins, that is, having
    $|\sigma_j - \sigma_i|$ small whenever $i$ is a neighbor vertex of $j$, does not prevent the spins
    from taking very large values. In order to avoid this, one introduces a
    parameter that penalizes large values of the spins.

\begin{mydef}[see \cite{DGFF2}]
The Hamiltonian of the Discrete Gaussian Free Field (GFF) model in $\Lambda \subset \mathbb{Z}^d$ is defined by
\begin{equation}
    H_{\Lambda, \beta,m}(\sigma)
=
    \frac{\beta}{4d}
    \sum_{\{i,j\} }
    (\sigma_i - \sigma_j)^2
    +
    \frac{m^2}{2}
    \sum_{i \in \Lambda} \sigma_i^2,
    \qquad \sigma \in \Omega,
\end{equation}
where $\beta \geq 0$ is the inverse temperature and $m \geq 0$ is the mass.
The model is called massive when $m > 0$, and massless when $m = 0$.
\end{mydef}

We can view the interaction term in the Hamiltonian comes from considering the transformation given by the difference operation.

\subsection{Energies and Gibbs measures linked with Arithmetics}

By analogy, we consider other versions of these by considering other transformations which are derivations or derivations-like in order to keep the theme. 
As in the previous sections, the transformations can be indeed more general. 
We can view this as a model of both APS (see below) and ALS as this is an element that one can view as a probabilistic model for which one  has proved a lot of results in the Probability community (see \cite{DGFF2}, \cite{DGFF1})  or can view this formulation as something close in spirit to the Energy-based models.
Energy-Based models (EBMs), which are probabilistic structures studied for many years that model distributions
by associating unnormalized probability densities. Such methods have been developed in many works in the past approximately 45 years (Hopfield, 1982 \cite{EBMhistory1}; Ackley et al., 1985 \cite{EBMhistory2}; Cipra, 1987 \cite{EBMhistory3}; Dayan et al., 1995 \cite{EBMhistory4}; Zhu et al., 1998 \cite{EBMhistory4};
Hinton, 2012\cite{EBMhistory5}). This perspective is unified  in the Machine Learning study under the name of Energy-based Models (EBMs) in \cite{EBMhistory6}. For more on this topic, see also \cite{bishop2023deep}. Gibbs measures are fundamental structures of interest both in Learning Theory and in Probability Theory.

We consider  in the following by analogy a model similar to this in which the difference is replaced with p-derivations or with mixes of usual derivations and p-derivations. This appears naturally when one searches for polynomials which belong to symbolic powers of certain ideals.  
One can think that the previously described Discrete GFF corresponds to a Dirichlet energy based on the difference operation. In the next subsection, we introduce an energy in which these operations are replaced with mixes of usual derivations and p-derivations. 
As mentioned, there are various results proved for the Discrete Gaussian Free Field  as this was well-studied in Probability Theory (see, for example, \cite{DGFF1}, \cite{DGFF2}) that will be interesting to study in the context of the new random field based on usual derivatives and p-derivations.

This energy can be used when one is interested in sampling polynomials with certain arithmetic properties. Indeed, following (\cite{bouman2022foundations}) one is interested in generating objects sampled from the Gibbs distribution. 
For sampling, we either restrict to a finite class of polynomials (for example, bounded degree and bounded coefficients), or add a coefficient penalty so that the corresponding partition function is finite. We consider the former case in the following.

\subsubsection{Arithmetic Energies}

Let us consider a prime number $p, $ and a loss function $\ell(\cdot)$. For a polynomial $f \in \mathbb{Z}[x]$ and ideals of the form $q_i=(p, x-a_i)$, we define the following Energy that is linked with the interpolation arithmetic information in the polynomials, that we can just call for simplicity 
'Arithmetic Energy' of level $L$ on $n$ ideals.

\begin{mydef}[Arithmetic Energy: one variable case]
    
    \begin{equation}
  \label{eq:energy}
  E^{L}_{q_1,\ldots, q_n}(f) \;=\; \sum_{i=1}^{n}\sum_{b_1+b_2 \leq L-1} 
    \ell \left(\delta_{st, p}^{b_1}\left(\frac{\partial^{b_2}f}{\partial x^{b_2}}\right) \text{mod} q_i\right)
\end{equation}

\end{mydef}
When the remainder is composed of constants, one choice is $l(\cdot)=(\cdot)^2$, for general formats of the remainders one can consider $\ell(\cdot)=||\cdot||^2$ when the remainder polynomial is not a constant. 

This arithmetic energy measures the discrepancy between  polynomial which respects all the mixed-derivatives conditions  (which will have zero energy) and the ones which fail a partial subset or the full set of those conditions (the energy increases with the number of conditions failed). 

In addition, one can consider variations of this by considering different exponents for various orders of contact to penalize differently higher order versus lower order contact but the idea is the same. We focus on this uniform case for convenience of the reader.

\begin{equation}
  \label{eq:energy}
  E^{L}_{q_1,\ldots, q_n}(f) \;=\; \sum_{i=1}^{n}\sum_{b_1+b_2 \leq L-1} 
    \left(\delta_{st, p}^{b_1}\left(\frac{\partial^{b_2}f}{\partial x^{b_2}}\right) \text{mod} q_i\right)^2
\end{equation}

This energy is zero if and only if $f$ along with its mixed-derivations belong to the ideals $q_i =
(p, x - a_i)$.  We note that as before, one can also consider more general prime ideals.
We note that compared with the loss function approach when we aimed at fitting some fixed data points and we were considering a search over the possible ideals, in the case of the energy-based method the perspective changes in the sense that the list of ideals is fixed and we search for polynomials which satisfy the corresponding conditions. This perspective helps us identify the polynomials which satisfy all or a part of the conditions and also to see how do neighboring polynomials behave in terms of this type of energies. Indeed, we aim to study this energy perspective as it not just provide us with the rigid class of polynomials that satisfy the conditions but also gives us information on neighboring polynomials of a certain polynomial and what happens in terms of the energy values (measuring how far is a certain object from satisfying the conditions) when one for example considers random perturbations of  polynomials. As usual, this is a particular choice, but one can consider other choices natural for the corresponding studies.

One can also define a version of this energy for multi-variables polynomials $g(x_1,\ldots, x_N) \in \mathbb{Z}[x_1, \ldots, x_N]$.

\begin{mydef}[Arithmetic Energy: Multi-variables case]

\begin{equation}
E^{L,N}_{q_1,\ldots, q_n}(g)=\sum_{i=1}^n\sum_{b_1+j_1+\ldots+j_N\leq L-1}\ell\left(\delta^{b_1}_{s t, p}\left(\frac{\partial^{j_1+...j_N}}{\partial x_1^{j_1} \ldots \partial x_N^{j_N}} g\left(x_1, \ldots, x_N\right)\right) \text{mod} \hspace{1mm} q_i\right)
\end{equation}
\end{mydef}

With the choice $\ell(\cdot)=||\cdot||^2$, this becomes 
\begin{equation}
E^{L,N}_{q_1,\ldots, q_n}(g)=\sum_{i=1}^n\sum_{b_1+j_1+\ldots+j_N\leq L-1}||\delta^{b_1}_{s t, p}\left(\frac{\partial^{j_1+...j_N}}{\partial x_1^{j_1} \ldots \partial x_N^{j_N}} g\left(x_1, \ldots, x_N\right)\right) \text{mod} \hspace{1mm} q_i||^2
\end{equation}

In addition, as before one can consider variations of this by considering different exponents for various orders of contact to penalize differently higher order versus lower order contact.

These energies also appear naturally when is interested in sampling. The question of sampling from Gibbs distributions is of interest in Machine Learning (see \cite{murphy2012machine}). Sampling from a probability distribution with this form makes the configurations of coefficients of integer coefficients polynomials which satisfy part of all of the conditions more likely. We note that there is no control over the coefficients size of the polynomials in this format of the energy but one can add another squared norm penalty term if one is interested also to control the size of the coefficients.

\begin{rmk}
As before, one can consider variations of this conditions involving derivatives and p-derivations such as the outputs of those to be on certain sub varieties or to have a certain format and study those variations of the energy. Indeed this version of the energy is designed to check the closeness of a polynomial to be in a certain symbolic power but one might  be interested in other patterns such as the derivatives or p-derivations having a certain prescribed format (say some to look like $Lx+b$ for some $L$ and $b$ after reducing mod p or other general formulas). 
\end{rmk}

We study some properties of the energy functions in the following.
\begin{Proposition}
The energy $E^L_{q_1, \ldots, q_n}$ has the following properties
\begin{itemize}
    \item[] 1)  $E^L_{q_1, \ldots, q_n}(f)=0$ if and only if f satisfies all the constraints. 

    \item[] 2) For fixed $L\geq 1$ and for a fixed selection of ideals $q_1, \ldots, q_n$, if $\mathcal{R}\subset \mathcal{R}'$ then $$\min_{f  \in \mathcal{R'}}E^L_{q_1, \ldots, q_n}(f)\leq \min_{f  \in \mathcal{R}}E^L_{q_1, \ldots, q_n}(f).$$

    \item[] 3) For fixed $\mathcal{R}$, and a fixed selection of ideals $\{q_1, \ldots q_l\}$ for any $f \in \mathcal{R}$, and for $L_1 \leq L_2$, we have $$E^{L_1}_{q_1, \ldots, q_l}(f)\leq E^{L_2}_{q_1, \ldots, q_l}(f).$$

    \item[] 4)  For fixed $L \geq 1$ and for fixed $\mathcal{R}$, if $Q=\{q_1, \ldots q_l\} \subset \tilde{Q}= \{q_1,\ldots, q_l, q_{l+1},\ldots, q_n\}, $ then for any $f \in \mathcal{R}$, we have $$E^L_{q_1, \ldots, q_l}(f)\leq E^L_{q_1, \ldots, q_n}(f).$$

\end{itemize}
\end{Proposition}

\begin{proof}
The proof of 1) and 2), and 3) follows similarly as in the proof of Proposition \ref{proploss}.  

For 4), as $\mathcal{R}$ is fixed, we have that for a set containing more ideals, we might add more positive terms to the energy if the conditions are not satisfied, thus, $E_{Q} \leq E_{\tilde{Q}}. $
\end{proof}
In addition, we explore the energy landscape. 

\begin{Proposition}
Using the computations performed in Theorem \ref{result1}, and summing over the contributions for each ideal $q_k$, $k \in \{1, \ldots, n\}$ we obtain that
\begin{align}
&E_{q_1, \ldots, q_n}^L(f_0+r)
\nonumber\\
&=\sum_{k=1}^n \sum_{j_1+\ldots+j_N\leq L-2}||\left(\delta_{s t, p}\frac{\partial^{j_1+...j_N}}{\partial x_1^{j_1} \ldots \partial x_N^{j_N}} f_0\left(x_1, \ldots, x_N\right)+ \delta_{s t, p}\frac{\partial^{j_1+...j_N}}{\partial x_1^{j_1} \ldots \partial x_N^{j_N}}r(x_1, \ldots, x_N)\right)\text{mod} \hspace{1mm} q_k\nonumber\\
&-\sum_{i=1}^{p-1} \frac{\binom{p}{i}}{p} (\frac{\partial^{j_1+...j_N}}{\partial x_1^{j_1} \ldots \partial x_N^{j_N}}f_0(x_1, \ldots, x_N))^i (\frac{\partial^{j_1+...j_N}}{\partial x_1^{j_1} \ldots \partial x_N^{j_N}}r(x_1, \ldots, x_N))^{p-i} \text{mod} \hspace{1mm} q_k||^2+\mathcal{E}(b_1 \neq 1),
\end{align}

\begin{align}
&E^L_{q_1,\ldots, q_n}(f_0+r^\tau)
\nonumber\\
&=\sum_{k=1}^n\sum_{j_1+\ldots+j_N\leq L-2}||\left(\delta_{s t, p}\frac{\partial^{j_1+...j_N}}{\partial x_1^{j_1} \ldots \partial x_N^{j_N}} f_0\left(x_1, \ldots, x_N\right)+ \delta_{s t, p}\frac{\partial^{j_1+...j_N}}{\partial x_1^{j_1} \ldots \partial x_N^{j_N}}r^\tau(x_1, \ldots, x_N)\right)\text{mod} \hspace{1mm} q_k\nonumber\\
&-\sum_{i=1}^{p-1} \frac{\binom{p}{i}}{p} (\frac{\partial^{j_1+...j_N}}{\partial x_1^{j_1} \ldots \partial x_N^{j_N}}f_0(x_1, \ldots, x_N))^i (\frac{\partial^{j_1+...j_N}}{\partial x_1^{j_1} \ldots \partial x_N^{j_N}}r^\tau(x_1, \ldots, x_N))^{p-i} \text{mod} \hspace{1mm} q_k||^2+\mathcal{E}(b_1 \neq 1).
\end{align}

\end{Proposition}

In the experiments, below the polynomials are excluded from the landscape if $E(f) = 0$ (exactly in all three ideals),
$f$ is the zero polynomial,
$f$ is a non-zero constant,
all coefficients of $f$ are divisible by $p$ (trivially zero energy mod~$p$).



\begin{figure}[H]
        \centering
        \includegraphics[width=.45\linewidth]{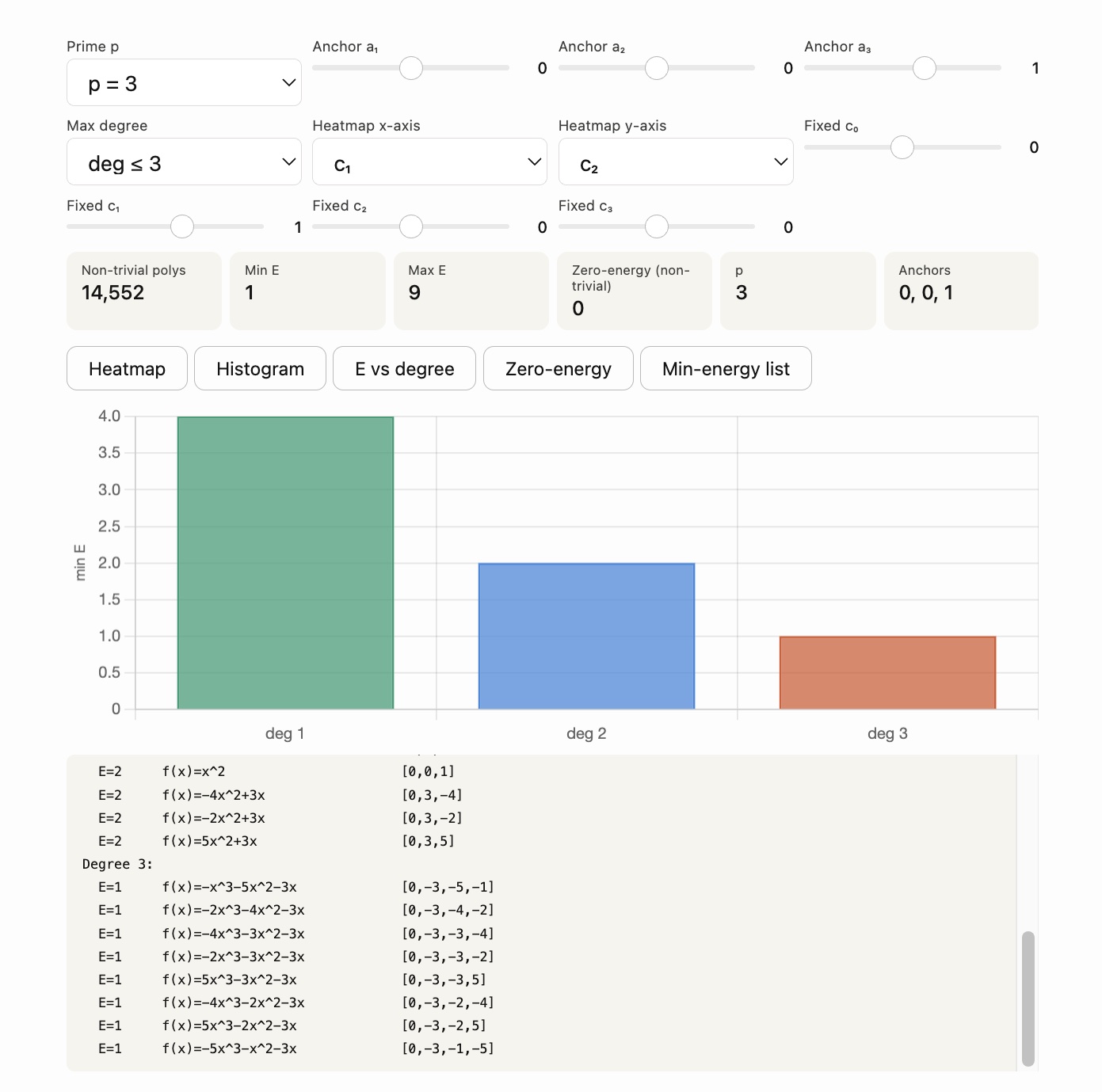}
        \caption{Example of energies when we require that the coefficients of the polynomials to not be divisible by the corresponding prime for two different ideals and p=3}
\end{figure}

\begin{figure}[H]
        \centering
        \includegraphics[width=.45\linewidth]{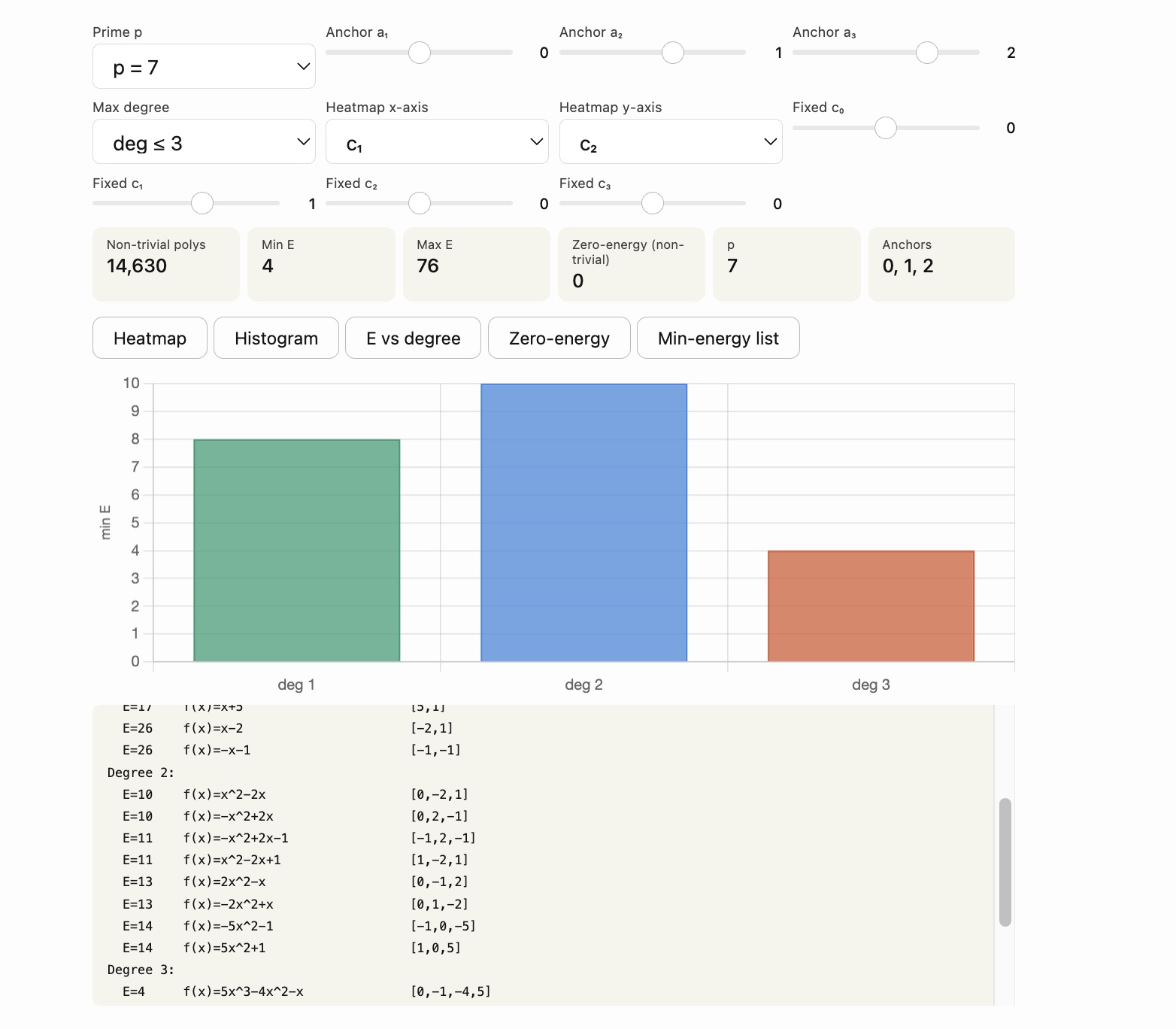}
        \caption{Example of energies when we require that the coefficients of the polynomials to not be divisible by the corresponding prime for three different ideals and p=7}
\end{figure}

As these configurations are rare (see also the experiments), these interact with Large Deviations (see \cite{LDPinstatmech} for Large Deviation Principle for Statistical Mechanics Models) as well as given the discreteness of the data with Hastings Metropolis algorithm in order to generate these types of Gibbs distributions.

These Gibbs distributions are also interesting in order to test the landscape of these energies around the minimizes. For example, how far away a perturbation of a polynomial puts it further from a local minimum. In addition, these exploration algorithms allow a study of the landscape around the local minimizer compared with a direct computation on the congruences. As we are on a discrete space, one can study in this context a Metropolis--Hastings algorithm to discover local minima, with the Metropolis transition rule
\[
\alpha(x,y)=\min(1,e^{-\Delta E}).
\]

This reads as 
\begin{itemize}
\item[] 1)If $E(y)<E(x)$, then $\Delta E<0$, so the move is accepted with probability $1$;
\item[] 2)If $E(y)>E(x)$, then $\Delta E>0$, so the move is accepted with probability $e^{-(E(y)-E(x))}$.
\end{itemize}

We refer the reader to \cite{Metropolishastingschapter8inthis} Chapter 8 for a description of this algorithm in general. One can try also a parameter $\beta$ in front of the energy as an analogue of an 'inverse temperature' and see if there are any phase transitions in this parameter. We consider the Metropolis–Hastings algorithm as a discrete-space alternative for exploring the energy landscape, in contrast to gradient-based methods on continuous spaces.  

\begin{rmk}
In addition, one might study also other formats of the energy in which the exact multiplicity of prime p at a certain point $x=a$ might be replaced with other conditions such as $f(a)=m_1a+m_2$ for some $m_1$ and $m_2$. Thus, the polynomial might not be exactly required to be a multiple of a prime p when evaluated at the point $x=a$ but might be having other structure. Similar conditions can be phrased for higher order derivatives or mixed of derivatives and p-derivations and the corresponding energies studied for those configurations. 

\end{rmk}

\begin{rmk}
We note that one can also consider a combination of usual derivatives and the p-derivations to obtain information about singularities. In \cite{Jacobiancriterionmixed} the authors obtain information on the singular ideals using a mix of usual derivatives and p-derivations. See also Theorem 3 in \cite{AMSdifbyprime} for an expository version.  Thus, one can also rephrase these problems in terms of some loss function on certain values, but now the loss function is built using derivatives of the polynomials along with the p-derivations and one looks on points and computes how these are further from zero by further adding the condition that the points are on the variety.
Indeed, one has another element that needs to be considered
$$
(\delta_p, \partial_{x_i}),\hspace{2mm} i \in\{1, \ldots, N\}.
$$
 \end{rmk}

This suggests in addition to derivative information to consider a \(p\)-derivation component, which plays the role of an 'arithmetic' direction in these discrete problems. This shows that some new transformations linked with the features of the discrete space need to be included in the gradient descent methods for certain classes of problems such as the ones linked with singularities in these spaces.

\section{Axiomatic Probabilistic Structures}

Linked with the previous part, in the second part of the paper, we introduce the framework Noisy Derivations Equations (NDEs). The terms RDEs, and SDEs are already reserved in the community. One can work with various notions of derivations. For example, one can consider the forward difference. That is, for \(h>0\), we consider difference operator $
\Delta_h f(x) := f(x+h) - f(x).$
Then, this will correspond to the study random dynamical systems which appears already in the literature. One can consider various models of derivations and consider random perturbations of these models. 
To keep the same theme with the sections before, we focus on the mixed cases with usual and p derivations, but this approach is more general and can include various derivations models. One motivation for this is also linked with the choice of the set of ideals $M_p$ where we test for roots mod p.  We want to see if knowing some extra data like the largest coefficient of the polynomial (or the order of it) can give us better choices  for what ideals to consider in $M_p$ for testing to be as efficeint as possible. Then, we investigate the following: if we search for roots at random at this scale, what will be the chance to find polynomials with these properties in our set of search? In addition, what are is the probability that solutions, if they exist, behave in a certain way, etc. 
We focus on Noisy Partial Derivations Equations (NPDes). We study these models with the following choice of derivations. The first derivation is the standard $\frac{d}{dx}$ derivative that it is linear and satisfies the Leibniz Rule. The second one is the p-derivation introduced before.  We emphasize that the framework in \cite{Buiumoriginalpaper1}, \cite{BuiumarithmeticPDE1}, is deterministic. We extend this to random cases.

\subsection{Motivation}

In the previous sections, we considered deterministic search of integer points. In this section, we consider random search (that is, a random selection of points). We would like to study if this random type search is successful. That is, what is do solutions of these problems exist, if so what is the chance that they match some partially observed data. One can investigate many questions, for example, what is the chance that these solutions are behaving in a certain way, etc. 
We first study the existence of solutions in the one-variable case. These types of equations appear when one considers random search at various scales which depend on characteristic of solutions.
One first model consists of random equations with noise of the form $\delta^{b_1}_{st,p}\frac{\partial^{b_2}}{\partial x}f \in (p, x-U(\omega))$, for non-negative integers $b_1$ and $b_2$. Then the right hand side would not depend directly on the solution $f$. This type of random equations correspond to random selections of the fibers without any interaction with the solution. In the models which we analyze in the following section, the random selections of the fibers depend on the degree of solutions, thus for given degrees of polynomials in the search space, the selections of anchors depends on that degree.

When we consider points that depend on the solution itself, we study for simplicity of the exposure the cases in which the scales are given by the degree of the polynomial.

\subsection{Extensions to NPDEs in the Arithmetic case}

In \cite{BuiumarithmeticPDE1} and \cite{BuiumarithemticPDE2}, \cite{SurveyarithmeticPDS} the authors describe extensions of the theory to the partial differential case. There are various choices for analogues of $(\mathbb{C}[x_1, x_2] , \partial_{x_1}, \partial_{x_2})$. One of them is the following
$(\mathbb{Z}[q], \delta_{p}, \delta_{q})$ with

$$\delta_p, \delta_q: \mathbb{Z}[q] \to \mathbb{Z}[q]$$

$$\delta_p(\sum_n a_nq^n):=\frac{(\sum_n a_nq^n)-(\sum_n a_nq^n)^p}{p},$$

$$\delta_q(\sum_n a_nq^n):=q\frac{d}{dq}(\sum_n a_nq^n)=\sum_n na_nq^n.$$

This is the so called case (1+2). In \cite{SurveyarithmeticPDS}, the author considers other cases as well. 

One can think about the directions of the standard derivative and the p-derivations as two different directions, one being the standard spatial one and the other one being an 'arithmetic' direction. Then, NPDEs appear when one is interested to find polynomials which have higher order contact conditions at points who are randomly distributed at a scale given by a characteristic of the solution (in our experiments, we considered the degree, but other variants are possible). These equations can also be thought as analogies with the non-singular ones considered in various works so far (see, for example,  \cite{carmona1995stationary}, \cite{godin2013lyapunov}, \cite{safronov136eigenvalue})

\subsection{Analogy with Random ODEs/PDEs}

In this subsection, we consider the analogy with Random ODEs/PDEs. We will construct an example on which all mixed operators obtained by first taking usual derivatives and then $p$-derivations, up to total order L, belong to $(p,x-a)$ uniformly in a. We will use the notion of p-adic valuation $v_p(\cdot),$ which for a fixed prime p, for given integer outputs the exponent of that prime in the  decomposition of the integer.

\subsubsection{Existence of solutions in the one-variable case}

For convenience of the reader, in this section, we focus on the existence of solutions in one-variable case. This analysis can be replicated for $N$-variables. 

\begin{Proposition}
Let $p$ be a prime, and let $\delta_p$ be the standard $p$-derivation on $\mathbb{Z}[x]$ associated to the Frobenius lift $
\phi(x)=x^p, \delta_p(f)=\frac{\phi(f)-f^p}{p}.$
Let
$
g(x)=p^k x^d \in \mathbb{Z}[x], k,d\geq 0.$
Then, for every pair of nonnegative integers $ b,m,$ one has
\[
\delta_p^m\!\bigl(\partial_x^b g\bigr)=a_{b,m}\,x^{(d-b)p^m}
\]
for some coefficient \(a_{b,m}\in \mathbb{Z}\), with the convention that \(\partial_x^b g=0\) if \(b>d\). Moreover, if $
b+m\leq L$
and $k\geq L+1,$
then $
\delta_p^m\!\bigl(\partial_x^b g\bigr)\in (p,x-a).$
for every $a\in\mathbb{Z}$. 
\end{Proposition}

\begin{proof}
We first study what happens when we apply the usual derivative. If $b>d$, then $
\partial_x^b g=0,$
and therefore $
\delta_p^m(\partial_x^b g)=0\in (p)\subseteq (p,x-a)$
for every $m\ge 0$ and every $a\in\mathbb Z$.

Let us assume now that $0\le b\le d$. Then, we have that
$
\partial_x^b g
=
p^k\, d(d-1)\cdots(d-b+1)\,x^{d-b}.$
We then write $
\partial_x^b g = c_b x^{d-b}, 
c_b:=p^k\, d(d-1)\cdots(d-b+1)\in \mathbb Z.$
For the integer $d(d-1)\cdots(d-b+1)$, we  have
$v_p(c_b)\ge k$. Let us apply iterated p-derivations to \(c_b x^{d-b}\). We claim that for every \(m\ge 0\), we have 
$\delta_p^m(c_b x^{d-b}) = a_{b,m}x^{(d-b)p^m}$
for some \(a_{b,m}\in\mathbb Z\) satisfying $
v_p(a_{b,m})\ge k-m. $ We show this by induction on \(m\) in the following.

For \(m=0\), $
\delta_p^0(\partial_x^b g)=\partial_x^b g = c_b x^{d-b},$
so we may take $
a_{b,0}=c_b,$
and indeed $
v_p(a_{b,0})=v_p(c_b)\ge k. $

Let us assume now that
$
\delta_p^m(\partial_x^b g)=a_{b,m}x^{(d-b)p^m}
$
with $
v_p(a_{b,m})\ge k-m.$
Then, by the definition of \(\delta_p\), we have 
$
\delta_p^{m+1}(\partial_x^b g)
=
\delta_p\!\left(a_{b,m}x^{(d-b)p^m}\right)
=
\frac{\phi\!\left(a_{b,m}x^{(d-b)p^m}\right)-\left(a_{b,m}x^{(d-b)p^m}\right)^p}{p}.$
Then, we have $
\delta_p^{m+1}(\partial_x^b g)
=
\frac{a_{b,m}x^{(d-b)p^{m+1}}-a_{b,m}^p x^{(d-b)p^{m+1}}}{p}
=
\frac{a_{b,m}-a_{b,m}^p}{p}\,x^{(d-b)p^{m+1}}.
$
Thus, we have $
a_{b,m+1}=\frac{a_{b,m}-a_{b,m}^p}{p}.$

We now estimate its p-adic valuation. If $k-m\ge 1$, then $
v_p(a_{b,m})\ge 1$. We thus obtain that $
a_{b,m}\equiv 0\pmod p$
so
$
1-a_{b,m}^{\,p-1}\equiv 1\pmod p.$
Therefore, we have that
$
a_{b,m}-a_{b,m}^p=a_{b,m}(1-a_{b,m}^{\,p-1})$
has the same \(p\)-adic valuation as $a_{b,m}$, and thus $
v_p(a_{b,m+1})
=
v_p\!\left(\frac{a_{b,m}-a_{b,m}^p}{p}\right)
=
v_p(a_{b,m})-1
\ge k-m-1
=
k-(m+1).$
This proves the induction. Hence, for every \(m\ge 0\), $
\delta_p^m(\partial_x^b g)=a_{b,m}x^{(d-b)p^m}$
with
$v_p(a_{b,m})\ge k-m.$
Now assume
$
b+m\le L$ and
$k\ge L+1.
$
Then \(m\le L\), so
$
k-m\ge k-L\ge 1.
$
Therefore
$
v_p(a_{b,m})\ge 1,$
which implies that every coefficient of \(\delta_p^m(\partial_x^b g)\) is divisible by \(p\). Hence $
\delta_p^m(\partial_x^b g)\in (p).$
Since
$ (p)\subseteq (p,x-a) $
for every \(a\in\mathbb Z\), it follows that $
\delta_p^m(\partial_x^b g)\in (p,x-a)
$
for every \(a\in\mathbb Z\), uniformly in a. This completes the proof.
\end{proof}

\begin{rmk}
Following the previous result, we obtain the following. Let \(L\ge 0\). If $g(x)=p^k x^d$
 \text{with}  
$k\ge L+1$
then for every pair of nonnegative integers \(b,m\) with
$ b+m\le L, $
one has
$
\delta_p^m\!\bigl(\partial_x^b g\bigr)\in (p,x-a)$
for every \(a\in\mathbb Z\).
\end{rmk}

\begin{rmk} 

In particular, for $b=0,$ one obtains the analogy with random ODEs (as there are just p-derivations involved, no mixes).
For $g(x)=p^k x^d \in \mathbb{Z}[x], k,d\geq 0,$, for every integer \(m\geq 0\) with \(m\leq k\), one has
$$
\delta_p^m(g)(x)=a_m x^{dp^m}
$$
for some \(a_m\in \mathbb{Z}\) satisfying $
v_p(a_m)=k-m.$

\end{rmk}

\section{Random search}

We study a model in which the search window scales
with the degree and perform random search with this state. In this model we are thus using a
state variable that is linked to root  count as there are relations between the degree of a polynomial and the number of roots.

\begin{mydef}[Degree-scaled random search]
For an integer-coefficient polynomial $f$ of degree $d=\deg(f)$ and a random variable which gives equal chance to integer points on $\{-3,\ldots,3\}$, define the random integer points
$$
  a(f,U) \;=\; d\cdot U.
$$

Thus, we consider the following:
$$\delta^{b_1}_{st,p}\frac{\partial^{b_2}}{\partial x}f \in (p, x-deg(f)U(\omega)),$$ for non-negative integers $b_1$ and $b_2$  such that $b_1+b_2 \le1 .$In the experiments below we consider another choice of the multiplication in front of the random variable on the right hand side that is $D+1-deg(f)$, for some $D \geq deg(f)+1.$ We investigate solutions in a certain subspace of polynomials with bounds on the coefficients.

For the anchors $a$ previously described, given that $b_1+b_2 \leq 1$, the ideal membership conditions are
$$
  C_1: f(a)\equiv 0,\quad C_2: f'(a)\equiv 0,\quad
  C_3: \delta_p(f)(a)\equiv 0\pmod{p}.
$$
\end{mydef}

Although we do not pursue this direction in the current work, one can study the cases
$$\delta^{b_1}_{st,p}\frac{\partial^{b_2}}{\partial x}f \in (p, x-deg(f)U(\omega)),$$ 
with $b_1+b_2 >1$, on certain subspaces of the integer-coefficients polynomials with bounded coefficients or beyond. One can see our initial study as analogues with Random ODEs and the latter direction mentioned as analogues with Random PDEs.

\subsection{Claude Experiments}

First, we consider the problem for fixed fibers $U=u_i, i<p-1$. For each fixed fiber $u_i$, we consider the set of solutions $\text{Sol}(u_i).$ As these equations experience non-uniqueness of solutions in general, we implement a selection rule among these solutions on each fixed fiber.  Although there are other models possible for selection, we consider the uniform choice among the solutions $\text{Sol}(u_i),$ for each $i \in \{1, \ldots, n\}$, with $n \leq p-2.$ That is, we select a solution from the corresponding fiber, with probability 
\begin{equation}
\mathbb{P}_i(\{f(x)\})=\frac{1}{|\text{Sol}(u_i)|},
\end{equation}
for each $i\leq p-2,$ whenever $|\text{Sol}(u_i)|$ is non-zero.

In order to compute or estimate this probability one is required to solve the system of equations over $\mathbb{Z}/p\mathbb{Z}$ given by a number of congruences to be satisfied based on the order of the p-derivations and usual derivations considered. Then, for each fixed fiber $u_i$, one can consider independent selections with the probability $\mathbb{P}_i.$  

We study $\frac{|Sol(u_i)\cap R_1|}{|{Sol{(u_i)}}|} $ across fibers. For $f=c_0+c_1 x+c_2 x^2$ with $a=deg(f)U(\omega) $,
we have 
\begin{align}
  C_1&:\quad c_0 + c_1 a + c_2 a^2 \equiv 0 \pmod{p}, \label{eq:lin1}\\
  C_2&:\quad c_1 + 2c_2 a \equiv 0 \pmod{p}, \label{eq:lin2}\\
    C_3&:\quad \delta_p(f)(a)\equiv 0\pmod{p}. \label{eq:lin3}
\end{align}

\begin{figure}[H]
        \centering
        \includegraphics[width=.95\linewidth]{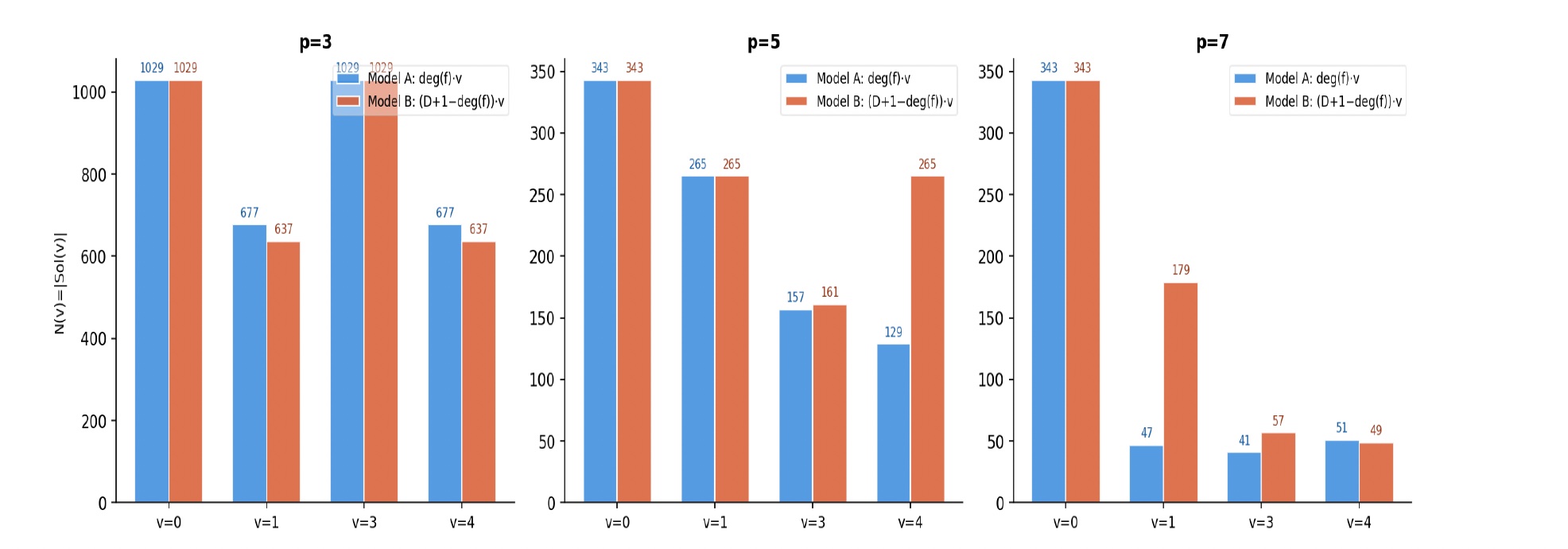}
        \caption{Number of solutions for various fibers for two choices of the right hand-side. We consider integer-coefficients polynomials with degree up to 4 and with integer coefficients between [-3, 3]. v is a random variable which assigns equal chance $1/4$ to $\{1, 2,3, 4\}.$}
\end{figure}

These experiments show that although following the previous results one can always construct a solution if the coefficients are of the form $p^k$ for some $k\geq m+1$ then there are multiple solutions also when coefficients are between  $[-3, 3] \cap \mathbb{Z}$ so less than $p^k$ for $k \geq 2$ in our case less than 9, 25, and 49. So these solutions are not directly due to the fact that the coefficients are automatically divisible by p. We note that we might see these patterns due to the fact that we work with one-variable and with first level conditions (the NDE case, the analogy with Random ODEs). 

\begin{rmk}
We note that we might see so many solutions as we study the problem in one dimension only where there are many combinations that solve the corresponding modular equations. For many variables and more restrictive equations the situation might be very different. 

\end{rmk}

\subsection{Claude experiments in which one has access to the data on few points}

Let us consider a prime p. We assume that we have partial observations of the data at some points. Due to computational power availability, we consider the scenario consisting of only one point, but in general can be a finite collection of points. We observe noisy data on those points. Then, one is interested to find out if there is any polynomial that belongs to a symbolic power of a certain order that is matching the observed data.
We perform some tests of this type of prototype random equations for small primes (due to the computational capabilities available) and show how likely are we to get matching data at a point.

\begin{figure}[H]
        \centering
        \includegraphics[width=.95\linewidth]{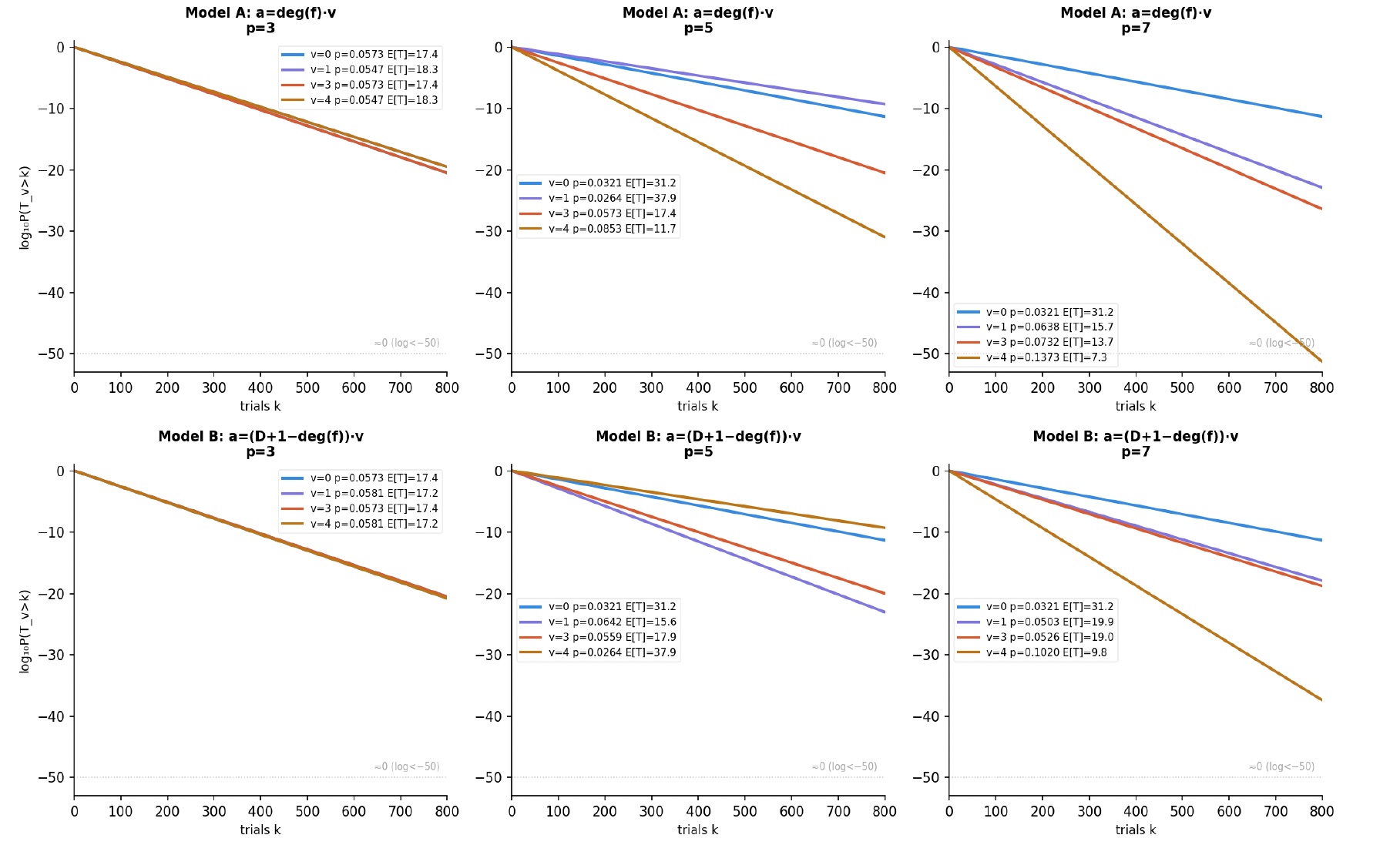}
        \caption{Survival probabilities for the number of experiments needed until first success for various fixed fibers.}
\end{figure}

\begin{rmk}
One can consider more general models with access to data on more points. 
Let $R_1$ and $R_2$ be two events that capture the fact that the restrictions $R_1$ and $R_2$ are satisfied. That is, $R_1=\{f(x_1) \in [\rho_1^1, \rho_1^2]\}$ and $R_2=\{f(x_2) \in [\rho_2^1, \rho_2^2]\}$, for some bounds $\rho_i^j$, $i \in \{1,2\}$, $j \in \{1, 2\}.$  One can consider a more general situation in which one has m region bound events $R_j, j \in \{1, \ldots, m\}.$
\end{rmk}

We observe from our experiments that the solutions to these types of equations with noise are in general not unique. There is literature which deals with non-uniqueness of solutions in the random setting, see for example \cite{cardona2020semiflow-nonuniqueref} and the references therein. The work in \cite{cardona2020semiflow-nonuniqueref} also considers solution sets, differential inclusions , etc. which are interesting perspectives to consider in our context as well in the future.  In our setting, we consider a selection mechanism that has maximal entropy in the sense that we pick uniformly among the solutions in the solution set on a fixed fiber. Other possible selection mechanisms among the solutions are possible and deserve a detailed study in the future.
In addition, it is also interesting to study the solution sets themselves.

\subsection{Extensions}

1)So far, we considered the uniform model (in the sense that values to be picked are equally likely) but one can consider other probabilistic models. \\
2) In addition, one can consider other formats and more variables. In this case, we considered a model in which the information which appears on the solution on the right-hand side of these equations is through the degree of it. One can consider also other functions of the solutions. In addition, we studied only the one-variable case but one can extend this analysis also to multi variables cases and ideals generated by a  prime and an irreducible random polynomial whose random coefficients depend on the solution.
One can for example consider equations with noise of the form

$$\delta_p \partial_x f(x, y, \omega) \in (p, x-\alpha(deg (f))\zeta_p(\omega) y^2).$$
3) Moreover, one can use Extreme Value Statistics, one can study one-sided versions. Then, one can compute and estimate probabilities of such events on each fiber for $l\geq 1$ independent experiments $\mathbb{P}_i(\max_jf_j(x_0) \leq R)=\prod_{j=1}^l\mathbb{P}_i(f_j(x_0) \leq R)=\mathbb{P}_i(f_j(x_0) \leq R)^l.$ Also,  $\mathbb{P}_i(\min_jf_j(x_0) \geq \tilde{R})=\prod_{j=1}^l\mathbb{P}_i(f_j(x_0) \geq \tilde{R})=\mathbb{P}_i(f_j(x_0) \geq \tilde{R})^l.$

In order to estimate the probability of these types of events one needs to estimate 
$\frac{|Sol(u_i)\cap \{f_j(x_0)\leq R\}|}{|Sol(u_i)|}$ for each fiber $u_i$. 
In order to compute or to estimate such probabilities, one needs estimates and theory from the area of Arithmetic Geometry as it relies on solution counts, etc. That is one of the reasons this direction of research is interesting as it relied on both arithmetic estimates and on probabilistic techniques and the combination between them. One can study more general events.






We note that, although we do not pursue this here as it is  a first work in this direction, a similar analysis can be carried for polynomials which are expected to have more roots. For example, one might consider two ideals and the random case in which the right hand-side does not depend on the solution.  

\begin{align}\label{rde1}
\delta_{st, p}^{b_1}\left(\frac{\partial^{b_2}}{\partial x^{b_2}} f(x;\omega)\right) \in &\left(p, x-x_1(\omega)\right)\cap\left(p, x-x_2(\omega)\right) .
\end{align}

For some functions $\sigma_1(\cdot)$, and $\sigma_2(\cdot)$ of the solutions, one can consider more involved equations such as

\begin{align}\label{rde2}
\delta_{st, p}^{b_1}\left(\frac{\partial^{b_2}}{\partial x^{b_2}} f(x;\omega)\right) \in &(p, x-\zeta_p(\omega)\sigma_1\left(f(x;\omega)\right))\cap\left(p, x-\zeta_p(\omega)\sigma_2\left(f(x;\omega)\right)\right).
\end{align}

Moreover, one can consider also dynamical version of the equations with noise considered so far.

For $f_0 \in \mathbb{Z}[x]$, let us consider

\begin{align}
\delta^{b_1}_{st, p}\left(\frac{\partial^{b_2}}{\partial x^{b_2}} f_{n+1}(x;\omega)\right) \in &\left(p, x-\zeta_{n,p}(\omega)\sigma_1\left(f_n(x;\omega)\right)\right).
\end{align}

These equations appear when considered coupled with the static ones. Indeed, imagine that you know that another investigator that is searching for these polynomials with these properties has tried the anchor at the scale of the degree in the previous step with a given noise. For another class of polynomials with higher degrees then one is interested to see if there can be solutions at the same scale as the one used for the static version by the other investigator. Indeed, when the degrees are large these scales might be pretty well separated and then the $(n+1)^{th}$ degree search might be on a considerably larger region.

One can also think about the dynamical version in the following steps. One first samples a random variable, then it considers the random search in the space of polynomials with certain characteristics, say the degree. Then it multiplies the random variable with a function of  a feature of the search space of the partner investigator, say the degree of polynomials used by them. Then searches for solutions to the samples scaled realizations of that random variable in the corresponding space (say polynomials with a degree higher, etc). 

The study of the structure of the noise  is an important topic in Probability theory. In our picture, the structure of the noise (say the extreme behavior/ the tails or the typical behavior etc.) interacts also with the arithmetic structure. Also, it is interesting to study techniques of Large Deviations Theory this direction. As before, one can consider extensions of these problems to multi-variables polynomials.







\section{Further applications of ALS}

The methods developed so far are suited, among others, also for studying data that are collected at a certain instance of a later time and deduce using the rules information about initial configurations. This can be useful for Economics, Biological Data, etc. 

We focused these initial explorations in the previous sections on ALS based on derivations which come naturally in the dataset of objects that are integer-coefficients polynomials. This approach can be extended to study of Economic models, for example the ones on graphs (see \cite{degroot1974reaching}, \cite{cerreia2024dynamic}, \cite{golub2026spectral}). In this context, one is interested to model the opinions on certain goods. 
This problem can be phrased both as a data fitting problem or as an energy problem.

\subsection{Other versions and applications of ALS in cases when the rules are not known}

The ALS approach allows to study the corresponding losses for various choices of axioms and obtain various properties and loss landscapes as functions of those rules. In  Economics, one can study also loss functions across many transformations simultaneously and their corresponding axioms. 

 The approach we explore is not restricted to Economics, and it can be replicated to Biology etc, in general to fields in which we want to discover the elements of a potential theory that can explain the patterns. In addition, using this approach one can be closer to a mathematical thinking in which one does not only get information from the solution itself but also from the other routes that lead to bigger errors as approaches that did not match.
 We note that one can consider the analysis over various choices of axioms for the transformations. In addition, these transformations might not be based on derivatives or derivations and is a setting in which more general transformations such as those based on updates rules are natural. One can study for each choice of the axioms the behavior of the models or one can study over a class of axioms simultaneously the behavior of the model. This goes along the lines of studying models in which the rules of update are not known in advance and one studies various scenarios individually in case these updates tend to be in a certain class of transformations or another. This formulation is particularly suitable for certain studies in which one does not have a precise choice yet and this is to be discovered.

Compared with Biology or other types of data sets that one can measure, the Economics models might not have an exact rule of updates, or these rules might change due to various factors so considering many potential models and study them one  by one or simultaneously looks more appealing for Economics and related disciplines with the same features. In Economics, the mechanism can also change with institutions, incentives, information, or the social environment. So a natural theoretical direction is to study how does the economic outcome change across plausible updating rules.

 \subsection{ALS in Economics: A  simplified initial model}

 Let us consider a model where $x_i<1$ and $x_i>1,$ with the interpretation of favoring a certain output and respectively being against it.
Let us model the Opinion Dynamics on trees which models the hierarchical structure of a company, etc. Let us consider the initial opinions at the leaves $x_i^{(0)}$ which will propagate through some mechanisms given by the family of transformations towards the root of the tree. 
We can investigate the following problem at any level of the tree between the leaves and the root. 
Let $N(i)$ denote the number of branches that a node $i$ has on the tree. A general update could be

$$
x_i^{(l+1)}
=
T_i(x_j^{(l)}), j\in N(i).
$$

For example, a weighted averaging rule could be

$$
T_i(x_j^{(l)})=\sum_{j\in N(i)}
w_{ij}x_j^{(l)},
\qquad
\sum_{j\in N(i)}w_{ij}=1.
$$
This is weighted opinion transformation but one can consider other general transformations.

Then the data-fitting loss at level $l$ in the tree is defined by 
$$
L^l_{\mathrm{data}}(T, x_0)
=
\sum_i
|
y_i-x_i^{(l)}
|^2,
$$
where $y_i$ is the observed opinion data at level $l$. One can study these quantities individually for choices of transformations Economics meaningfully. In addition, one can consider this formulation uniformly over many transformations possible to be considered and check the disagreement from a fixed initial opinion or it can also consider a fixed transformation and then change the initial opinion at the leaves (for example, considering  $x_0 +y_0$, for some $y_0$ on each opinion on the leaves) and use the axioms/rules of those transformations to see how it propagates up to any level $l$. One can consider various other perturbations of the initial opinions on leaves. 

One can consider also more involved models in which one checks the consensus at level $l$ among the opinions formed. Then, one can also formulate an energy version of this. Overall, there are many interesting problems to be considered. 

Indeed, one can consider

$$L_I^l(T, x_0)
=
\inf_{T\in\mathcal T} \sum_i
|
y_i-x_i^{(l)}
|^2,$$ as well as 

$$L_S^l(T, x_0)
=
\sup_{T\in\mathcal T}\sum_i
|
y_i-x_i^{(l)}
|^2,$$
and perform an analysis over a family of transformations with Economics meaning.

In addition, one can consider an Energy approach to measure closeness to consensus,  to the level $l$  $$ E_l= \sum_i V(x_i^{(l)}) +\sum_{\{i,j\}} W(x^{(l)}_i,x^{(l)}_j).$$ 
A potential choice for the elements in this energy can be $V(x) = \kappa(x-1)_+^2$. This choice can favor opinions below the action threshold \(1\), and $W(x^{(l)}_i,x^{(l)}_j) = (x^{(l)}_i-x^{(l)}_j)^2$ can favor agreement. One starts with opinions profile at the level of leaves and obtains the values at level $l$ through the rules of the transformations. Thus, one can use various initial opinions profile and perturbations of those as before and study through the axioms of the transformations various energy profiles. We note that this is an initial model used to exemplify the ideas. A similar approach can be considered for more involved models.

\section{Discussion, open problems, other examples of ALS and APS}

\subsection{ALS and APS: at cross-roads of many research areas}

The current development sits at the cross-roads between several disciplines such as Learning Theory, Probability, Optimization, to mention a few. Through the study of the interaction with the axioms of the objects and transformations under consideration it links naturally with many fields of research.

We believe that, in general, techniques from the areas of Probability Theory (Asymptotic results, Concentration Inequalities, etc.) Large Deviations, Optimal Transport, Control Theory, etc. will be beneficial for these studies. We also mention the other way around: these areas will bring new problems to study in those areas which study the interactions of those techniques with arithmetic properties, derivations-like transformations, general derivations, etc.

Let us consider one of the directions and give more details. We consider the interaction in particular with Optimization Theory in the study of $ALS_d$. In the one direction, we study Arithmetic problems with Optimization techniques (considering an loss function functional and its minimization) in the context of noisy data driven polynomial recovery and Mathematical objects data-sets for generating conjectures, testing new rules, etc.  That captures the use of Optimization techniques in other areas.

In the opposite direction, these types of problems give new perspective on the study of Optimization problems where the loss function functions contain other choices of derivations (for example the ones we consider studied in Arithmetic Geometry) rather than the usual derivative. in terms of the fundamental functions on the space such as polynomials, operations between them such as addition etc, and natural p derivations and other operations that are designed through he axioms. Axioms are needed as there are more types of calculus.



\subsection{Other examples of ALS and APS, non-exhaustive list of open problems}

\begin{itemize}

\item \textbf{Random polynomials, Random fields and Geometry}

We discuss briefly other structures of interest in Probability Theory and their interactions with the axioms, namely Random Fields. We naturally want to consider them as elements of an Axiomatic Probabilistic Structure in which we consider instead the standard derivative we consider, to keep the theme, p-derivations or mixed derivations and p-derivations. One aim of this future direction is to explore the interaction of this structure of interest in Probability with the axioms of the transformations involved. 

One can consider various probabilistic model on polynomials of degree $d\leq D$ with integer coefficients. There are great monographs devoted to the study of Random Fields and their geometry, see for example \cite{Randomfieldsandgeometry}. In \cite{Randomfieldsandgeometry} the authors consider the study of random functions, the statistics of their critical points, the excursion, sets, etc.
The critical points of the random smooth functions are obtained via $\nabla f=0$ for $f$ a random smooth function. In addition, can do a similar analysis by replacing the operations before with mixes of usual and p-derivations and studies their zeros mod p, for some prime p. The interpretation would be then different as then one obtains for these random polynomials, sets of their higher contact behavior. As usual, one can consider also other models.


\item \textbf{A model for a random metric in the primes coordinates setting}

Another potential direction is to consider random analogues of the Arithmetic Differential Geometry developed, between others, in the monograph \cite{Arithmeticgeometrybook}. This contains a deterministic study of a metric model in which one considers prime numbers as the coordinates and the p-derivation as the substitute of the usual derivatives (see \cite{Arithmeticgeometrybook}). 
Random Metric models are structures of interest in Probability Theory, so this is another example of APS. One can consider a random metric model in which the usual derivatives are replaced by the p-derivations and a bilinear form with random coefficients, say on the diagonal, that model the chance that the p-derivation is applied for that prime p to the integer coefficient polynomial. Then one can check again for the obtained polynomial after this operation by resampling the metric, Other models are possible. Then, one can study the evolution of an integer-coefficient polynomial under a sequence of transformations  and ask for example survival probability (\cite{junk2014survival}) type questions linked with divisibility properties after n steps and phrase 'survival' in terms of that. In addition, one can replace p-derivations with some general transformations for other models.

\item Given the flexibility of the framework, one can replicate ideas in other directions. These directions can be of both theoretical and practical interest. As these are initial applications of these ideas it will be interesting to study the interactions between the axioms (for instance those that have to do with arithmetic) and other topics covered in \cite{BishopML} such as Information Theory, Decision Theory,etc. In addition, given the structures already studied in text data (see \cite{benchmarking}), future work could investigate whether our Transformation-Axiomatic based approach provides additional structural insights.

\end{itemize}

\textbf{Acknowledgements:}
I would like to thank Gabriel Terejanu and Razvan Bunescu for interesting discussions on the topic of Machine Learning. I would like to thank my spouse Laura for the schematic drawings used in this paper. 

\bibliographystyle{plain}
\bibliography{ref}

\end{document}